\documentclass[a4paper,12pt]{article}
\usepackage{amssymb}
\usepackage{amsmath,amssymb,amsfonts}
\usepackage{graphicx}
\usepackage{epsfig}
\usepackage{multirow}
\usepackage{mathrsfs}
\usepackage{amsfonts}
\usepackage{dsfont}
\usepackage{graphicx}
\usepackage{subfigure}
\usepackage[colorinlistoftodos]{todonotes}
\usepackage[colorlinks=true, allcolors=blue]{hyperref}
\usepackage{amsthm}
\usepackage{float}
\usepackage{hyperref}
\usepackage{float}
\theoremstyle{plain}
\usepackage[english]{babel}
\usepackage[utf8x]{inputenc}
\usepackage[T1]{fontenc}
\newtheorem{theorem}{Theorem}[section]
\newtheorem{lemma}{Lemma}[section]
\newtheorem{definition}{Definition}[section]
\newtheorem{remark}[theorem]{Remark}

\newtheorem{corollary}{Corollary}
\numberwithin{equation}{section}
\numberwithin{lemma}{section}

\title{Limited Aperture Inverse Scattering Problems using Bayesian Approach and Extended Sampling Method
\thanks{The research of Li and Deng was supported in part by NNSF of China under grant 11771068.
}}
\author{Z. Li
\thanks{School of Mathematical Sciences, University of Electronic Science and Technology
of China, Chengdu, 611731, China. ({\tt lzx130682@163.com}).}
\and Z. Deng
\thanks{School of Mathematical Sciences, University of Electronic Science and Technology
of China, Chengdu, 611731, China. ({\tt ldengzhl@uestc.edu.cn}).}
\and J. Sun
\thanks{Department of Mathematical Sciences, Michigan Technological University, Houghton, MI 49931, U.S.A. ({\tt jiguangs@mtu.edu}).}
}
\date{}
\begin{document}
\maketitle
\begin{abstract}
Inverse scattering problems have many important applications. In this paper, given limited aperture data, we propose a Bayesian method for the inverse acoustic scattering to reconstruct the shape of an obstacle. The inverse problem is formulated as a statistical model using the Baye's formula. The well-posedness is proved in the sense of the Hellinger metric. The extended sampling method is modified to provide the initial guess of the target location, which is critical to the fast convergence of the MCMC algorithm. An extensive numerical study is presented to illustrate the performance of the proposed method.
\end{abstract}

\section{Introduction}
Inverse scattering problems have important applications such as radar, medical imaging, and non-destructive testing.
The goal is to detect and identify the unknown object using acoustic, electromagnetic or elastic waves, etc.
\cite{Baum1999, ColtonKress2013}. Depending on how much data can be obtained, 
the inverse scattering problems can be categorized as the full aperture problems and the limited aperture problems.

In the context of the inverse scattering theory,
many methods have been proposed for the full aperture inverse scattering problems
\cite{CakoniColton2014, ColtonKirsch1996IP, Kirsch98}.
These methods usually provide satisfactory reconstructions.
However, for a lot of practical applications such as underground mineral prospection and visually obscured target detection, 
it is not possible to measure the full aperture data \cite{Baum1999} and thus only limited aperture data are available.
There exist relatively less literatures on the limited aperture inverse scattering problems
\cite{AhnJeonMaPark, BaoLiu2003SIAMSC, IkehataNiemiSiltanen, OchsJr1987SIAMAM, ItoJinZou2012IP, 
Zinn1989, ChengPengYamamoto2005IP, AmmariEtal2012SIAMIS}.
In an early work  \cite{Lewis1969IEEEAP}, Lewis proposed a simple method to reconstruct the shape of the target based on an integral identity. 
Later works such as  \cite{Zinn1989, BaoLiu2003SIAMSC, OchsJr1987SIAMAM, AmmariEtal2012SIAMIS} used the framework of shape optimization. 
The range test, direct sampling methods, extended sampling method, etc., were also proposed to process the scattering data of one incident wave 
\cite{Potthast2006IP, ItoJinZou2012IP, LiuSun2018IP}.
An alternative approach is to obtain the full aperture from the limited aperture measurements. 
Analytic continuation, a severely ill-posed problem, was considered by some researchers \cite{Atkinson1978, ChengYamamoto1998IP, ChengPengYamamoto2005IP}.
In some works \cite{Juhl2008JASA, LiuSun2019JCP}, the full aperture data was recovered using some integral equations 
together with regularization schemes. Then the methods for full aperture data can be applied.
Other researchers take the approach by modifying the classical sampling methods using full-aperture data for the limited aperture problems.
The uniqueness of the inverse problems can be proved in some cases \cite{ColtonKress2013}.
The reconstruction is not as good as the full aperture case in general \cite{GuoMonkColton2016AA, AudibertHaddar2017SIAMIS}.

Recently, the Bayesian framework has received increasing attention for inverse problems \cite{Fitzpatrick1991IP, KaipioSomerdalo2005, Stuart2010AN, AmmarEtal2013}. 
The inverse problem is recasted in the form of statistical inferences.
Variables are modeled as being random and the known information is coded in the priors. 
Using the Bayes' formula, the solution to the inverse problem becomes the posterior probability distribution of the unknown quantities.
We refer the readers to \cite{KaipioSomerdalo2005, Stuart2010AN} on the Bayesian framework for inverse problems and 
\cite{BaussardEtal2001IP, BuiGhatts2014SIAMUQ, YangMaZheng2015, HarrisRome2017AA, LiuLiuSun2019IP} on its applications to some inverse scattering problems.

In this paper, we focus on the development of a Bayesian method for the limited aperture inverse scattering problem
to reconstruct the boundary of a sound soft obstacle. 
The inverse problem is reformulated as a statistical quest of information.
The well-posedness is proved and an  MCMC algorithm is proposed to explore the posterior probability distribution.
It is critical to know the location of the target for the convergence of the MCMC algorithm. 
Recently, a new method, called the extended sampling method (ESM), 
was developed to obtain the size and shape of the target using the scattering data of one incident wave \cite{LiuSun2018IP, LiuLiuSun2019SIAMIS}. 
We modified the ESM such that it can be used to process limited aperture scattering data such that the location of the target can be obtained effectively using
the same set of measurement data.

The rest of the paper is organized as follows. In Section 2, we introduce the
direct scattering problem for a sound soft obstacle and the limited aperture inverse scattering problem. An integral approach is introduced for the direct scattering problem. 
In Section 3, we propose a modified ESM to obtain the obstacle location.
Section 4 contains the Bayesian formulation for the inverse problem.
Gaussian priors are used for the boundary parametrization of the obstacle, whose covariance operator is the inverse of the Laplacian. 
We provide a stability analysis for the Bayesian posterior probability distribution of the unknown shape parameters with respect to the noises.
In Section 5, we develop an efficient MCMC algorithm to explore the posterior probability distribution.
In Section 6, numerical examples are presented to validate the effectiveness of the proposed method.
Finally, in Section 7, we draw some conclusions and discuss future works.

\section{Direct and Inverse Scattering Problems}
Let $\Omega \subset \mathbb R^2$ be a bounded, simply connected domain with $C^2$ boundary $\partial\Omega$.
Denote by $\nu$ the unit outward normal to $\partial \Omega$.
Define  $\mathbb S = \{ x \in \mathbb R^2, |x| =1\}$.
The incident plane wave with direction $d \in \mathbb S$ is given by
\[
u^{i}(x):={\rm e}^{i kx\cdot d}, \quad x \in \mathbb{R}^2, d \in \mathbb S,
\] 
where $k>0$ is the wavenumber.
The direct scattering problem is to find  the scattered field $u^s$, or the total field $u=u^i+u^s$, such that
\begin{subequations}\label{5a}
\begin{align}
\label{5}\triangle u+k^2 u=0,& \quad \text{in}~~\mathbb{R}^2\backslash\overline{\Omega},\\[1mm]
\label{6a1}u=0,& \quad\text{on}~~\partial\Omega,~~~\\[1mm]
\label{7}
\lim\limits_{r\rightarrow\infty}\sqrt{r}\left(\frac{\partial u^s}{\partial r}-iku^s \right)=0.&
\end{align}
\end{subequations}
Equation~\eqref{6a1} is the sound-soft boundary condition and \eqref{7} is the Sommerfeld radiation condition. 
It is well-known that \eqref{5a} has a unique solution and the scattered field $u^s$ has an asymptotic expansion \cite{ColtonKress2013}
\[
u^s(x, d)
=\frac{e^{i \frac{\pi}{4}}}{\sqrt{8k\pi}}\frac{e^{i kr}}{\sqrt{r}}\left\{u^{\infty}(\hat{x}, d)+\mathcal{O}\left(\frac{1}{r}\right)\right\}\quad\mbox{as }\,r:=|x|\rightarrow\infty
\]
uniformly in all directions $\hat{x}=x/|x|$. The function $u^{\infty}(\hat{x}, d)$ is called the far-field pattern.

The limited aperture inverse scattering problem considered in this paper can be stated as follows.

{\bf LAIScaP:}  Determine $\partial \Omega$ from the far field pattern $u^{\infty}(\hat{x}, d), (\hat{x}, d) \in \gamma^o \times \gamma^i$, where
$\gamma^o, \gamma^i \subsetneq \mathbb S$ (see Fig.~\ref{Figure1}).

For example, if $\gamma^i=\{d\}$ and $\gamma^o=\mathbb S$, we have the far field pattern due to one incident wave.
If $\gamma^i=\gamma^o=\{d\}$, we have the back-scattering data.

In contrast, the full aperture problem is such that $u^{\infty}(\hat{x}; d)$ is available for all $\hat{x}, d \in \mathbb S$, i.e., $\gamma^i = \gamma^o = \mathbb S$.
It is well-known that the sound soft obstacle $\Omega$ can be uniquely determined by the full aperture far field pattern $u^\infty(\hat{x},d)$ for all $\hat{x}, d\in \mathbb{S}$.
Due to analyticity, the full aperture data  $u^\infty(\hat{x},d)$ for $(\hat{x}, d) \in \mathbb{S} \times \mathbb S$ is uniquely determined by
 $u^\infty(\hat{x},d)$ for $(\hat{x}, d) \in \gamma^o \times \gamma^i$ if both $\gamma^o$ and $\gamma^i$ are connected and have positive meansures.

\begin{figure}
\begin{center}
{ \scalebox{0.36} {\includegraphics{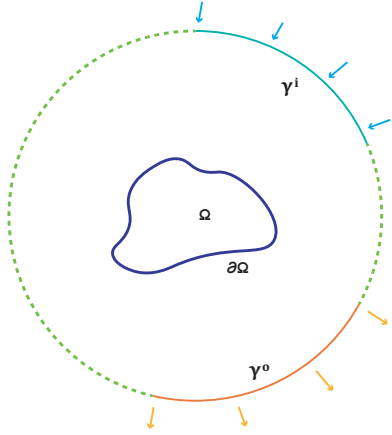}}}
\caption{The obstacle $\Omega$, the aperture of the incident waves $\gamma^i$, and the aperture of observation $\gamma^o$.}
 \label{Figure1}
\end{center}
\end{figure}

In the rest of this section, we present an integral equation formulation following \cite{ColtonKress2013} for the direct scattering problem \eqref{5a}.
The results will be used to analyze the Bayesian method and simulate the scattered fields in the MCMC algorithm.
Recall that the fundamental solution $\Phi(x,y)$ of the Helmholtz equation is given by
\begin{equation*}
\Phi(x,y)=\frac{i}{4}H^1_0(|x-y|),
\end{equation*}
where $H^1_0$ is the Hankel function of the first kind of order zero.

Define the single layer potential operator $S$ 
 \begin{equation}\label{18}
(S\varphi)(x)=2\int_{\partial\Omega}\Phi(x,y)\varphi(y)ds(y),\quad x\in\partial  \Omega,
\end{equation}
and the double layer potential operator $K$ 
\begin{equation}\label{19}
(K\varphi)(x)=2\int_{\partial\Omega}\frac{\partial\Phi(x,y)}{\partial\nu(y)}
\varphi(y)ds(y),\quad x\in\partial  \Omega.
\end{equation}
Then $S$ and $K$ are bounded from $C^{0,\alpha}(\partial\Omega)$ into $C^{1,\alpha}(\partial\Omega)$ (Theorem 3.4 of \cite{ColtonKress2013}).

Using the single and double layer potentials, one can write the scattered field as
\begin{equation}\label{15}
u^s(x;\Omega)=\int_{\partial \Omega}\left\{\frac{\partial\Phi(x,y)}{\partial \nu(y)}-i\xi \Phi(x,y)\right\}
\varphi(y)ds(y),\quad x\in \mathbb R^2 \setminus \overline{\Omega},
\end{equation}
where $\xi$ is a real coupling parameter and $\varphi(y)$ is the unknown density function.
Then the direct scattering problem is to find the density $\varphi$ such that
\begin{equation}\label{20}
(I+K-i\xi S)\varphi=-2u^i \quad \text{on } \partial \Omega.
\end{equation}
There exists a unique solution $\varphi$ satisfying \eqref{20} and depending continuously on $u^i$ (Theorem 3.11 of \cite{ColtonKress2013}).
Furthermore, the far field pattern can be written as (Page 80 of \cite{ColtonKress2013})
\begin{equation}\label{uinfty}
u_\infty(\hat{x}, d) = \frac{e^{-i\frac{\pi}{4}}}{\sqrt{8\pi k}} \int_{\partial \Omega} ( k \nu(y) \cdot \hat{x} + \xi ) e^{-ik \hat{x} \cdot y} \varphi(y) ds(y),
\end{equation}
where $\varphi$ is the solution of \eqref{20}.

\section{Extended Sampling Method for LAIScaP}
Given limited aperture far field data, we first consider the problem of finding the location of the obstacle $\Omega$.
Recently, a simple method, called the extended sampling method (ESM), was proposed 
using the the far field data due to one incident wave \cite{LiuSun2018IP, LiuLiuSun2019SIAMIS}.
The method can effectively reconstruct the location and size of the obstacle. 
In this section, we generalize the ESM for the limited aperture data to obtain the location of the obstacle,
which is of critical importance for the convergence of the MCMC algorithm. 

We first present the ESM for one incident wave briefly here and refer the readers to \cite{LiuSun2018IP} for details.
Assume that the far field pattern $u_\infty(\hat{x}, d_0)$ is available for one incident plane wave with direction $d_0$.
Let $B \subset \mathbb{R}^{2}$ be a disc centered at the origin with radius $R$ large enough. 
The far field pattern for $B$ corresponding to the incident plane wave with direction $d$ is given by (see, e.g., Chp. 3 of \cite{ColtonKress2013}):
\begin{flalign}\label{uinf}
     u^B_{\infty}(\hat{x}, d)=-e^{-i \frac{\pi}{4}}\sqrt{\frac{2}{\pi k}}\bigg[\frac{J_0(kR)}{H_0^{1}(kR)}+2\sum_{n=1}^{\infty}\frac{J_n(kR)}{H_n^{1}(kR)}\cos(n\theta)\bigg], 
     \quad  \hat{x}\in \mathbb{S},
  \end{flalign}
where $J_n$ is the Bessel function, $H_{n}^{1}$ is the Hankel function of the first kind of order $n$, $\theta =\angle (\hat{x},d)$, the angle
between $\hat{x}$ and $d$. Define
\[
B_z:=\{x+z| x\in B, z \in \mathbb{R}^2\}
\] 
and let $u_{\infty}^{B_z}(\hat{x}, d)$ be the far field pattern of $B_z$. Then the following translation property holds
\begin{equation}\label{shift}
u^{B_z}_{\infty}(\hat{x}, d)=e^{i kz\cdot (d-\hat{x})}u^B_{\infty}(\hat{x},d),\ \ \ \ \hat{x}\in\mathbb{S}.
\end{equation}

Let $T$ be a domain with $\Omega$ inside.  For $z \in T$, define a far field operator $\mathcal{F}_z: L^2(\mathbb{S}) \to L^2(\mathbb{S})$ such that
\begin{equation}\label{FO}
\mathcal{F}_z g(\hat{x}) = \int_{\mathbb{S}}u^{B_z}_{\infty}(\hat{x},y)g(y)d s(y), \quad \hat{x} \in \mathbb{S}.
\end{equation}
Using $\mathcal{F}_z$, we set up a far field equation
\begin{equation}\label{fe}
\left(\mathcal{F}_z g\right)(\hat{x})=u_\infty(\hat{x}, d_0),\ \ \ \ \hat{x}\in \mathbb{S}.
\end{equation}
This integral equation is the main ingredient of the ESM. 
\begin{theorem}\label{theorem1}(Theorem 3.1 of  \cite{LiuSun2018IP})
Let $B_z$ be a sound-soft disc centered at $z$ with radius $R$. Let $\Omega$ be an inhomogeneous medium or an obstacle with Dirichlet, Neumann, or the impedance boundary condition. Assume that $k$ is not a Dirichlet eigenvalue for $B_z$. Then the following results hold for the far field equation (\ref{fe}):
\begin{itemize}
\item[1.] If $D\subset B_z$, for a given $\varepsilon>0$, there exists a function $g_z^\varepsilon\in L^2(\mathbb{S})$ such that
\begin{equation}\label{fe2}
\bigg\|\int_{\mathbb{S}}u^{B_z}_{\infty}(\hat{x},d)g_z^\varepsilon(d)d s(d)-U_\infty(\hat{x})\bigg\|_{L^2(\mathbb{S})}<\varepsilon
\end{equation}
and the Herglotz wave function $v_{g_z^\varepsilon}(x):=\int_{\mathbb{S}}e^{ikx\cdot d}g_z^\varepsilon (d)ds(d), x\in B_z$ converges to the solution $w\in H^1(B_z)$ of the Helmholtz equation with $w=-U^s$ on $\partial B_z$ as $\varepsilon\rightarrow 0$.

\item[2.] If $D\cap B_z=\emptyset$, every $g_z^\varepsilon\in L^2(\mathbb{S})$ that satisfies (\ref{fe2}) for a given $\varepsilon>0$ is such that
\begin{equation}
\lim_{\varepsilon\rightarrow 0}\|v_{g_z^\varepsilon}\|_{H^1(B_z)}=\infty.
\end{equation}
\end{itemize}
\end{theorem}
Consequently, $\Omega$ can be reconstructed using the regularized solutions of \eqref{fe} for all the sampling points $z$'s in the domain $T$ of interrogation. 
The advantage of using $\mathcal{F}_z$ is that it can be computed ahead of time easily. In contrast, the classical far field operator uses full aperture far field data
and does not work for a single incident plane wave. While the location of $\Omega$ can be effectively determined, one can only obtain the location and rough
size of $\Omega$. Fortunately, this is enough for our purpose in this paper.

In the following, we generalize the above method for {\bf LAIScaP} to obtain the location of the obstacle.
We first consider the case of $u_\infty(\hat{x}, d)$ for a single incident plane wave with direction $d$ and $\hat{x} \in \gamma^o$, a non-trivial proper subset of $\mathbb{S}$. 
In fact, one can directly employ the ESM by solving the far field equations with the limited observation data 
\begin{equation}\label{feS0}
\mathcal{F}_z g(\hat{x}, d)=u^\infty(\hat{x}, d),\ \ \ \ \hat{x}\in \gamma^o.
\end{equation}
Note that due to analyticity of the far field pattern, Theorem~\ref{theorem1} holds when $\gamma^o$ contains a non-trivial connected subset of $\mathbb S$.
The indicator for the sampling point $z$ can be defined as
\begin{equation}\label{indicator}
I_z(d) = \|g_z^\epsilon(d)\|_{L^2}, \quad z \in T,
\end{equation}
where $g^\epsilon_z$ is the regularized solution of \eqref{feS0}. One can find the location of $\Omega$ by plotting $I_z$ for all $z \in T$.

For the general case of $u_\infty(\hat{x}, d), (\hat{x}, d)\in \gamma^o \times \gamma^i$, the indicator can be defined as
\begin{equation}\label{indicatordi}
I_z =\int_{\gamma^i} I_z(d) \text{d}s(d), \quad z \in T.
\end{equation}
In practice, the far field data are available for discrete sets of incident and observation directions, e.g.,
\[
u_\infty(\hat{x}_i, d_j), \quad \hat{x}_i \in \{\hat{x}_1, \ldots, \hat{x}_I\} \subset \mathbb S, \quad d_j \in \{d_1, \ldots, d_J\} \subset \mathbb S.
\]
For each $j$, as \eqref{feS0}, set up the equations
\begin{equation}\label{gzj}
\mathcal{F}_z g(\hat{x}_i, d_j)=u^\infty(\hat{x}_i, d_j), \quad i=1, \ldots, I,
\end{equation}
which is an ill-posed linear system. Let ${\bf g}_z^j$ be the regularized solution of \eqref{gzj}. Then the discrete indicator $I_z$ for multiple incident directions
is simply
\begin{equation}\label{Iz}
I_z = \sum_{j=1}^J |{\bf g}_z^j|.
\end{equation}
The ESM to obtain the location of the obstacle using limited aperture data is as follows.

\vskip 0.2cm
{\bf ESM for LAIScaP}
\begin{itemize}
\item[ ] input - $u_\infty(\hat{x}_i, d_j), \quad i=1, \ldots, I, j=1, \ldots, J$.
\item[1.] Generate a set of sampling points for a domain $T$ which contains $\Omega$.
\item[2.] Compute $u^{B_z}_{\infty}(\hat{x},d)$ for all $\hat{x}\in \mathbb{S}$ and $d \in \mathbb{S}$ for each $z \in T$.
\item[3.] For each observation direction $d_j$, set up a linear system according to \eqref{gzj} and compute an approximate solution ${\bf g}_z^j$.
\item[4.] Sum $|{\bf g}_z^j|$ over $j$ to obtain $I_z$ as \eqref{Iz}.
\item[5.] Find the minimum $I_{z_0}$ of $I_z$ and choose $z_0$ as the location of $\Omega$.
\end{itemize}


\section{Bayesian Approach}
The direct scattering problem can be written as
\begin{equation}\label{9}
u^\infty(\hat{x}, d)=\mathcal{F}(\Omega), \quad (\hat{x}, d) \in \gamma^o \times \gamma^i, 
\end{equation}
where $\mathcal{F}$ is the shape-to-measurement operator.
Assume that the boundary $\partial \Omega$ can be parametrized as
\begin{equation}\label{pOq}
\partial\Omega:=z_0+r(\theta)(\cos\theta,\sin\theta)
=z_0+\exp(q(\theta))(\cos\theta,\sin\theta), \quad \theta\in [0,2\pi),
\end{equation}
where $q(\theta)=\ln r(\theta), 0<r(\theta)<r_{\max}$, and $z_0$ is the location of $\Omega$.

Using the above parameterization and taking the noise in measurements into account, one can rewrite \eqref{9} as a statistical model
\begin{equation}\label{9a}
y=\mathcal{F}(q)+\eta,
\end{equation}
where  $q\in X$ and $y=u_\infty(\hat{x},d) \in Y$ for some suitable Banach spaces $X$ and $Y$. 
In particular, $y$ is the noisy observations of $u^\infty(\hat{x}, d)$ and $\eta(\hat{x}, d)$ is the noise.
In this paper, we assume that the observation noise is normal with mean zero and independent of $q$, i.e., $\eta(\hat{x}, d) \sim {\mathcal N}(0,\sigma^2)$.

In this paper, we choose $Y=\mathbb C$ and $X=C^{2,\alpha}[0,2\pi], \alpha \in (0, 1]$ \cite{BuiGhatts2014SIAMUQ}.
Define a norm on $\|\cdot\|_X$ as 
\begin{equation}\label{9a2}
\|q\|_X=\|q\|_\infty+
\|q'\|_\infty+
\|q''\|_\infty+
\sup_{\substack{\theta,\hat{\theta}\in[0,2\pi]\\ \theta\neq\hat{\theta} }}
 \frac{|q{''}(\theta)-q{''}(\hat{\theta})|}{|\theta-\hat{\theta}|^{\alpha}}.
\end{equation}
Assume that $q$ has the probability measure $\mu_0={\mathcal N}(m_0,c_0^2)$. 
We denote the posterior probability measure of $q$ by $\mu_y$. 
Let $\pi_0$ and $\pi_y$ denote the probability density functions of $\mu_0$ and $\mu_y$ respectively.
By Bayes' formula \cite{KaipioSomerdalo2005}, the posterior density function is
\begin{equation}\label{11}
\pi_y(q)=\frac{\pi_\eta\big(y-\mathcal{F}(q)\big)\pi_0(q)}
{\displaystyle\int_{X} \pi_\eta(y-\mathcal{F}(q))\pi_0(q)dq }.\\
\end{equation}
Thus
\begin{equation}\label{11a}
\begin{split}
\pi_y(q)&\propto\pi_\eta\big(y-\mathcal{F}(q)\big)\pi_0(q)\\
&\propto\exp\left(-\frac{1}{2}\left|\sigma^{-1}\big(y-\mathcal{F}(q)\big)\right|^2-
\frac{1}{2}\left\|c_0^{-1}(q-m_0)\right\|_X^2\right)\\
&\propto\exp\left(-\frac{1}{2}\left(|y-\mathcal{F}(q)|_\sigma^2+
\|q-m_0\|_{c_0}^2\right)\right),
\end{split}
\end{equation}
where $\propto$ means {\it is proportional to} 
and $\|\cdot\|_{c_0}= \|c_0^{-1}\cdot\|$ are covariance weighted norms.  
The inverse problem becomes the statistical inference of the posterior density $\pi_y(q)$.

 In the rest of this section, we study the well-posedness of the Bayesian method. We shall follow \cite{Stuart2010AN, BuiGhatts2014SIAMUQ}
 and extend the theory to the limited-aperture inverse scattering problems.
 Using the parametrization \eqref{pOq} for $\partial \Omega$ and results from \cite{ColtonKress2013}, we have the follow property for the 
 scattering operator $\mathcal{F}$.
 \begin{lemma}\label{LFq}
 For fixed $\hat{x}, d$ and every $\varepsilon>0$, there exists $M=M(\varepsilon)$ such that
 \begin{equation}\label{12}
|{\mathcal F}(q)|_\sigma\leqslant\exp\left(\varepsilon\|q\|^2_X+M\right)
\end{equation}
for all $q\in X$.
\end{lemma}
\begin{proof}
Plugging the parametrization \eqref{pOq} into \eqref{uinfty}, the far field pattern can be written as
\[
u_\infty(\hat{x}, d; q) = \frac{e^{i\frac{\pi}{4}}}{\sqrt{8\pi k}} \int_{0}^{2\pi} ( k \nu(\theta) \cdot \hat{x} + \xi ) e^{-ik \hat{x} \cdot (e^q\cos \theta, e^q \sin \theta)^T} \varphi(\theta) e^{q(\theta)}\sqrt{1+q'(\theta)^2}d\theta.
\]
Hence we have that
\begin{equation}\label{21}
|u^\infty(\hat{x}, d; q)|\leqslant C\left\|\sqrt{1+q'^2}\right\|_{\infty}  \exp\left(\|q\|_{\infty}\right), \quad (\hat{x}, d) \in \gamma^o \times \gamma^i.
\end{equation}
When $\|q'\|_{\infty}\leqslant1$, it is clear that
 \[
 \left\|\sqrt{1+q'^2} \right\|_{\infty}\leqslant C.
 \]
When $\|q'\|_{\infty}\geqslant 1$, according to Young's inequality, we have
\begin{equation}\label{22}
\left\|\sqrt{1+q'^2}\right\|_{\infty} \leqslant \sqrt{2} \left\|q'\right\|_{\infty}
\leqslant \sqrt{2} \exp\left(\|q\|_X\right)
\leqslant \sqrt{2} \exp\left(\frac{1}{2\varepsilon}+\frac{\varepsilon\|q\|^2_X}{2}\right).
\end{equation}
On the other hand, the following estimation holds
\begin{equation}\label{23}
\exp(\|q\|_{\infty})
\leqslant  \exp\big(\|q\|_X\big)
\leqslant  \exp\left(\frac{1}{2\varepsilon}+\frac{\varepsilon\|q\|^2_X}{2}\right).
\end{equation}
Substitution of \eqref{22} and \eqref{23} into \eqref{21} yields 
\begin{equation*}\label{24}
|u^\infty(\hat{x}, d; q)|
\leqslant  \exp\left(\varepsilon\|q\|^2_X+M\right), \quad (\hat{x}, d) \in \gamma^o \times \gamma^i.
\end{equation*}
Since $\gamma^o$ and $\gamma^i$ are bounded, we obtain \eqref{12} and the proof is complete.
\end{proof}

\begin{lemma}\label{Fq1Fq2}
For fixed $\hat{x}, d$ and every $\tau>0$, there exists $M=M(\tau)>0$, such that, for all $q_1,\;q_2\in X$ with $\max \{ \|q_1\|_X, \|q_2\|_X\}<\tau,$
\begin{equation}\label{1a}
|{\mathcal F}(q_1)-\mathcal{F}(q_2)|_{\sigma}\leqslant M\|q_1-q_2\|_{X}.
\end{equation}
\end{lemma}
\begin{proof}
Due to \eqref{uinfty}, we only need to show
\begin{equation*}
\|\varphi_1-\varphi_2\|_{\infty}\leqslant M\|q_1-q_2\|_{X},
\end{equation*}
which follows the proof of Theorem 5.16 of \cite{ColtonKress2013}.
\end{proof}


\begin{definition} 
The Hellinger distance between $\mu_1$ and $\mu_2$ with common reference measure $\nu$ is
\begin{equation}\label{34}
\rho_{_{H}}(\mu_{1},\mu_{2})=\left(\frac{1}{2}
\int\left(\sqrt{{d\mu_1}/{d\nu}}-\sqrt{{d\mu_2}/{d\nu}}\right)^2d\nu \right)^{1/2}.
\end{equation}
\end{definition}

Recall that if $\mu$ and $\nu$ are two measures on the same measure space, then $\mu$ is absolutely continuous with respect to
$\nu$ if $\nu(A)=0$ implies $\mu(A) = 0$ for $A \subset X$, written as $\mu \ll \nu$.
The Fernique Theorem (see, e.g., \cite{Stuart2010AN}) states the following.
If $\mu=\mathcal{N}(0,\sigma^2)$ is a Gaussian measure on Banach space $X$, so that $\mu(X)=1$, then there exists $\varepsilon>0$ such that
\begin{equation}\label{36a}
\int_X\exp(\varepsilon\|x\|^2_X)\mu(dx)<\infty.
\end{equation}

\begin{theorem}
Assume that $\mu_0$ is a Gaussian measure satisfying $\mu_0(X)=1$ and $\mu_y\ll \mu_0$. 
For $y_1$, $y_2$ with $\max\{|y_1|,~|y_2|\}<\tau$, there exists $M=M(\tau)>0$ such that 
\begin{equation}\label{37}
\rho_{_{H}}(\mu_{y_1},\mu_{y_2})\leqslant M|y_1-y_2|.
\end{equation}
\end{theorem}
\begin{proof}
For fixed $\hat{x}$ and $d$, $\mathcal{F}(q)=\mathcal{F}(\hat{x},d;q): X\to Y$ is a continuous map.
The Radon-Nikodym derivative is given by
\begin{equation*}\label{38}
\frac{d\mu_y}{d\mu_0}(q)=\frac{1}{A(y)}\exp\big(-\frac{1}{2}|y-\mathcal{F}(q)|^2_{\sigma}\big),
\end{equation*}
where
\begin{equation}\label{39}
A(y)={\int_{X}\exp\big(-\frac{1}{2}|y-\mathcal{F}(q)|^2_{\sigma}\big)d\mu_0(q)}.
\end{equation}
It is clearly that
\begin{equation}\label{39a}
A(y)\leqslant 1.
\end{equation}
From Lemma 4.1 and (\ref{39}), we have that
\begin{equation}\label{41}
\begin{split}
A(y)&\geqslant \int_{\{\|q\|_X\leqslant 1\}}\exp\big(-(|y|^2_{\sigma}+|\mathcal{F}(q)|^2_{\sigma})\big)d\mu_0(q)\\
&\geqslant\int_{\{\|q\|_X\leqslant 1\}}\exp(-M(\tau))d\mu_0(q)\\
&=\exp(-M(\tau))\mu_0\{\|q\|_X\leqslant 1\}\\
&>0,
\end{split}
\end{equation}
since the unit ball in $X$ has positive measure and $\mu_0$ is Gaussian.

Furthermore, using Lemma 4.1 and the Fernique Theorem, we have
\begin{eqnarray}
\nonumber |A(y_1)-A(y_2)|&\leqslant& \int_X \Big|\exp\big(-\frac{1}{2}\big|y_1-\mathcal{F}(q)\big|_{\sigma}^2\big)
-\exp\big(-\frac{1}{2}\big|y_2-\mathcal{F}(q)\big|_{\sigma}^2\big)\Big|d\mu_0(q)\\
\nonumber &\leqslant& \int_X \Big|\frac{1}{2}
\big|y_1-\mathcal{F}(q)\big|_{\sigma}^2-
\frac{1}{2}
\big|y_2-\mathcal{F}(q)\big|_{\sigma}^2\Big|d\mu_0(q)\\
\nonumber &\leqslant& \int_X \frac{1}{2} \Big||y_1|_{\sigma}^2-
|y_2|_{\sigma}^2\Big|+
\big|\mathcal{F}(q)\big|_{\sigma}|y_1-y_2|_{\sigma}d\mu_0(q)\\
\nonumber &\leqslant& \int_X \Big(\tau+
|\mathcal{F}(q)|_{\sigma}\Big)|y_1-y_2|_{\sigma}d\mu_0(q)\\
\nonumber &\leqslant& \int_X \exp(\varepsilon\|q\|_X^2+\ln [ \exp(M)+\tau\exp(-\varepsilon\|q\|_X^2)])\sigma^{-1}|y_1-y_2|d\mu_0(q)\\
\label{43} &\leqslant& M|y_1-y_2| .
\end{eqnarray}
From the definition of $\rho_H$, one has that
\begin{eqnarray}
\nonumber &&\rho^2_{_{H}}(\mu_{y_1},\mu_{y_2})\\
\nonumber&=&\frac{1}{2}
\int_X \left\{\left(\frac{\exp(-\frac{1}{2}|y_1-\mathcal{F}(q)|_{\sigma}^2)}{A(y_1)}\right)^{1/2}
-\left(\frac{\exp(-\frac{1}{2}|(y_2-\mathcal{F}(q)|_{\sigma}^2)}{A(y_2)}\right)^{1/2}\right\}^2
d\mu_0(q) \\
\nonumber&=&\frac{1}{2}
\int_X \left\{\left(\frac{\exp(-\frac{1}{2}|y_1-\mathcal{F}(q)|_{\sigma}^2)}{A(y_1)}\right)^{1/2}
-\left(\frac{\exp(-\frac{1}{2}|(y_2-\mathcal{F}(q)|_{\sigma}^2)}{A(y_1)}\right)^{1/2}\right.\\
\nonumber&&\qquad \left.+\left(\frac{\exp(-\frac{1}{2}|y_2-\mathcal{F}(q)|_{\sigma}^2)}{A(y_1)}\right)^{1/2}
-\left(\frac{\exp(-\frac{1}{2}|(y_2-\mathcal{F}(q)|_{\sigma}^2)}{A(y_2)}\right)^{1/2}\right\}^2
d\mu_0(q)\\
\nonumber&\leqslant& A(y_1)^{-1}\int_X  \left\{\exp\left(-\frac{1}{4}\big|y_1-\mathcal{F}(q)\big|_{\sigma}^2\right)
-\exp\left(-\frac{1}{4}\big|y_2-\mathcal{F}(q)\big|_{\sigma}^2\right)\right\}^2d\mu_0(q)\\
&&\qquad +\big|A(y_1)^{-1/2}-A(y_2)^{-1/2}\big|^2
\int_X \exp\left(-\frac{1}{2}\big|y_2-\mathcal{F}(q)\big|_{\sigma}^2\right)d\mu_0(q).\label{45}
\end{eqnarray}
Using Lemma 4.1 and the Fernique Theorem again, we have
\begin{eqnarray}
\nonumber&& \int_X \left\{\exp\left(-\frac{1}{4}\big|y_1-\mathcal{F}(q)\big|_{\sigma}^2\right)
-\exp\left(-\frac{1}{4}\big|y_2-\mathcal{F}(q)\big|_{\sigma}^2\right)\right\}^2d\mu_0(q)\\
\nonumber&\leqslant& \int_X \Big|\frac{1}{4}
\big|y_1-\mathcal{F}(q)\big|_{\sigma}^2-
\frac{1}{4}
\big|y_2-\mathcal{F}(q)\big|_{\sigma}^2\Big|^2d\mu_0(q)\\
\nonumber&\leqslant& \int_X \Big(\frac{1}{4} \Big||y_1|_{\sigma}^2-
|y_2|_{\sigma}^2\Big|+
\frac{1}{2}\big|\mathcal{F}(q)\big|_{\sigma}|y_1-y_2|_{\sigma}\Big)^2d\mu_0(q)\\
\nonumber&\leqslant& \int_X \frac{1}{4}\Big(\tau+|\mathcal{F}(q)|_{\sigma}\Big)^2|y_1-y_2|_{\sigma}^2d\mu_0(q)\\
\nonumber&\leqslant& \int_X \frac{1}{4}\exp(2\varepsilon\|q\|_X^2+2\ln [ \exp(M)+\tau\exp(-\varepsilon\|q\|_X^2)])\sigma^{-2}|y_1-y_2|^2d\mu_0(q)\\
&\leqslant& M|y_1-y_2|^2. \label{46}
\end{eqnarray}
According to the boundedness of $A(y_1)$ and $A(y_2)$, it holds that
\begin{eqnarray}
\nonumber&&\big|A(y_1)^{-1/2}-A(y_2)^{-1/2}\big|^2 \\
\nonumber&\leqslant& M\max\left(A(y_1)^{-3},A(y_2)^{-3}\right)
|A(y_1)-A(y_2)|^2\\
\label{51}&\leqslant& M|y_1-y_2|^2.
\end{eqnarray}
Combining \eqref{39a}-\eqref{51} we obtain\eqref{37}.
\end{proof}

For the limited aperture data $u^{\infty}(\hat{x}, d), (\hat{x}, d) \in \gamma^o \times \gamma^i$, the following result holds.
\begin{corollary}
\begin{equation}\label{cox}
\int_{\gamma^i} \int_{\gamma^o} \rho_{H}(\mu_{y_1(\hat{x},d)},\mu_{y_2(\hat{x}, d)}) \text{d}s(\hat{x}) \text{d}s(d)\leqslant M\|q_1-q_2\|_X.
\end{equation}
\end{corollary}
\begin{proof}
Due to the fact that $\gamma^o$ and $\gamma^i$ are bounded sets, \eqref{cox} follows Lemma \ref{Fq1Fq2} immediately.
\end{proof}

\section{Numerical Algorithm}
Now we present a Metropolis-Hastings MCMC method to generate samples to explore the posterior probability density \eqref{11a}.
Firstly one needs to choose a prior distribution for $q$. 
According to Lemma 6.25 of \cite{Stuart2010AN}, one could assume a Gaussian prior which is consistent with the above theory (see also \cite{BuiGhatts2014SIAMUQ}):
\[
q_a''(\theta)\sim {\mathcal N}(0,\mathcal{A}^{-s}),\quad s > \frac{1}{2},
\]
where $\mathcal{A}:=-d^2/d\theta^2$ with the definition domain
\begin{equation*}
D(\mathcal{A}):=\left\{ v\in H^2[0,2\pi]:\;\;\int^{2\pi}_0 v(\theta)d\theta=0\right\}.
\end{equation*}
The eigenvalues of $\mathcal{A}$ are $n^2,\,n\in \mathbb{N}$
and the corresponding eigenfunctions are $\phi_{2n}=\cos(n\theta)/\sqrt{\pi}$ and $\phi_{2n-1}=\sin(n\theta)/\sqrt{\pi}$. 
Karhunen-Lo\`{e}ve expansion implies
\[
    q''(\theta) = \sum^\infty_{n=1}\left(\frac{a_n}{n^s}
\frac{\sin(n\theta)}{\sqrt{\pi}}+\frac{b_n}{n^s}
\frac{\cos(n\theta)}{\sqrt{\pi}}\right),
\]
where $a_n$ and $b_n$ are i.i.d. (independent and identically distributed) with $a_n\sim {\mathcal N}(0,1)$ and $b_n\sim {\mathcal N}(0,1)$. 
Integrating $q''(\theta)$, we obtain
\[
    q(\theta) =
\frac{a_0}{\sqrt{2\pi}}-\sum^N_{n=1}\left(\frac{a_n}{n^{s+2}}
\frac{\cos(n\theta)}{\sqrt{\pi}}+\frac{b_n}{n^{s+2}}
\frac{\sin(n\theta)}{\sqrt{\pi}}\right).
\]
 Choosing $s=1$, we have that
\begin{equation*}
    q_a''(\theta) =
\sum^\infty_{n=1}\left(\frac{a_n}{n}
\frac{\sin(n\theta)}{\sqrt{\pi}}+\frac{b_n}{n}
\frac{\cos(n\theta)}{\sqrt{\pi}}\right),
\end{equation*}
where $a_n$ and $b_n$ are i.i.d. (independent and identically distributed) with $a_n\sim {\mathcal N}(0,1)$ and $b_n\sim {\mathcal N}(0,1)$. 
Integrating $q''(\theta)$, we obtain
\[
    q_a(\theta) =
\frac{a_0}{\sqrt{2\pi}}-\sum^N_{n=1}\left(\frac{a_n}{n^3}
\frac{\cos(n\theta)}{\sqrt{\pi}}+\frac{b_n}{n^3}
\frac{\sin(n\theta)}{\sqrt{\pi}}\right).
\]

Note that the choice of the prior distribution is not unique \cite{KaipioSomerdalo2005}. As the second choice. 
\begin{equation*}
    q_b'(\theta) =
\sum^\infty_{n=1}\left(\frac{a_n}{n}
\frac{\sin(n\theta)}{\sqrt{\pi}}+\frac{b_n}{n}
\frac{\cos(n\theta)}{\sqrt{\pi}}\right).
\end{equation*}
Integrating and differentiating $q'(\theta)$, we obtain
\begin{equation}\label{qtheta} 
    q_b(\theta) =
\frac{a_0}{\sqrt{2\pi}}-\sum^N_{n=1}\left(\frac{a_n}{n^2}
\frac{\cos(n\theta)}{\sqrt{\pi}}-\frac{b_n}{n^2}
\frac{\sin(n\theta)}{\sqrt{\pi}}\right).
\end{equation}
for some constant $a_0$ and
\begin{equation}\label{qprimeprime}
    q_b''(\theta) =
\sum^\infty_{n=1}\left({a_n}
\frac{\cos(n\theta)}{\sqrt{\pi}}-{b_n}
\frac{\sin(n\theta)}{\sqrt{\pi}}\right),
\end{equation}
respectively.

For the third choice, we take
\[
q_c(\theta) = \frac{a_0}{\sqrt{2\pi}}+\sum^N_{n=1}\left(\frac{a_n}{n} \frac{\cos(n\theta)}{\sqrt{\pi}}+\frac{b_n}{n} \frac{\sin(n\theta)}{\sqrt{\pi}}\right).
\]
As a consequence, one has that
\[
q_c'(\theta) = \sum^N_{n=1}\left(-{a_n} \frac{\sin(n\theta)}{\sqrt{\pi}}+{b_n}\frac{\cos(n\theta)}{\sqrt{\pi}}\right)
\]
and
\[
q_c''(\theta) = \sum^N_{n=1}\left(-n{a_n} \frac{\cos(n\theta)}{\sqrt{\pi}}+n{b_n}\frac{\sin(n\theta)}{\sqrt{\pi}}\right).
\]
Note that $q_b$ and $q_c$ do not satisfy Lemma 6.25 of \cite{Stuart2010AN}.

Secondly, a Markov proposal kernel is needed for the MCMC method. 
This kernel proposes moves from the current state of the Markov chain to the next state. 
The new state is then accepted or rejected according to a criterion using the target distribution $\mu_y$ of \eqref{11a}.
In this paper, the proposal kernel is chosen as
\begin{equation}\label{kernelq}
q \Leftarrow f_1(q)=\sqrt{1-2\beta}q+\sqrt{2\beta}\xi.
\end{equation}
 where $\xi\sim \mathcal{N}(0,1)$ and $\beta$ is a scale parameter. 
 Again, note that there are various choices of the proposal kernel \cite{Stuart2010AN}, e.g.,
 $$ q \Leftarrow f_2(q)=q+\sqrt{2\beta}\xi.$$
 For each state, to evaluate \eqref{11a}, one needs to solve the direct scattering problem \eqref{5a},
 which is done by the Nystr\"{o}m method (see, e.g., Section 3 of \cite{ColtonKress2013}). 

\vskip 0.2cm
{\bf  LAIScaP using MCMC and ESM}
\begin{enumerate}
\item Use ESM to obtain the location $z_0$ of the obstacle $\Omega$.
\item Set $N$ in expansion \eqref{qtheta} and number of iterations $K$.
\item Choose 
	\[
		a_0=1, a_1=a_2=\cdots=a_N=b_N=0.
	\]
\item for $k=2:K$ do
	\begin{itemize}
 	\item Calculate $q_1(\theta), q_1^\prime(\theta), q_1^{\prime\prime}(\theta)$ from \eqref{qtheta}. 
	\item Solve the direct problem \eqref{5a} for $F(q_1)$ and calculate $\pi_z$ from (\ref{11a}).
	\item Draw $a_0,a_1,b_1,\cdots, a_N,b_N$ from $\mathcal{N}(0,1)$, respectively. 
	\item Calculate $\tilde{q_1}(\theta), \tilde{q_1}{'}(\theta), \tilde{q_1}{''}(\theta)$ and set
		\begin{eqnarray*}
			q_2&=&\sqrt{1-2\beta}q_1+\sqrt{2\beta}\tilde{q}_1,\\
			q_2^\prime&=&\sqrt{1-2\beta}q_1^\prime+\sqrt{2\beta}\tilde{q}_1^\prime,\\
			q_2^{\prime\prime}&=&\sqrt{1-2\beta}q_1^{\prime\prime}+\sqrt{2\beta}\tilde{q}_1^{\prime\prime}.
		\end{eqnarray*}
	\item Solve \eqref{5a} and calculate  $\pi_z^\prime.$
	\item Calculate the acceptance ratio
    		$$\alpha(q_1,q_2)=\min(1,{\pi_z^\prime}/{\pi_z}).$$
	\item Draw $\tilde{\alpha}\sim U(0,1),$\\
\indent  if $\tilde{\alpha}<\alpha(q_1,q_2)$\\
 \indent ~~~~ accept and set $q_1(\theta)=q_2(\theta), q_1^\prime(\theta)=q_2^\prime(\theta), q_1^{\prime\prime}(\theta)=q_2^{\prime\prime}(\theta)$ \\
\indent   else\\
\indent  ~~~~ reject\\
\indent  end
	\end{itemize}
\item Compute the CM for the last $1000$ $q$'s.
\end{enumerate}

\begin{remark}
It is possible to assume that the location $z_0$ is unknown and satisfies certain priors. However, the computational cost is prohibitive and
the reconstruction is unsatisfactory. It is important to known the correct location of the obstacle.
\end{remark}


\section{Numerical  Examples}
We present some numerical examples to show the effectiveness of the proposed method.
The incident field is given by the time harmonic acoustic plane wave
\[
u^i(x)=e^{ikx\cdot d}, \quad d \in \mathbb S.
\]
We fix the wave number $k=1$ and set $N = 10$ in \eqref{qtheta} and assume that the corresponding coefficients $a_n$'s and $b_n$'s obey the same distribution $\mathcal{N}(0,1)$.
For all examples,  we take the last $1000$ samples to compute the conditional mean for $a_n$'s and $b_n$'s.

We choose two obstacles: a kite given by
\[
  (x_1, x_2) = (\cos \theta+0.65\cos2\theta-0.65, 1.5  \sin \theta) \quad \theta \in (0, 2\pi]
\]
and a pear given by
\[
 (x_1, x_2)=\left(\dfrac{5+\sin3\theta}{6}\cos\theta,\dfrac{5+\sin3\theta}{6}\sin\theta \right), \quad \theta \in (0, 2\pi].
\]


The limited aperture far-field data $u^{\infty}(\hat{x}, d), (\hat{x}, d) \in \gamma^o \times \gamma^i$ is computed by a linear finite element method.
Let $\phi$ be the observation angle such that $\hat{x}:=(\cos\phi,\sin\phi)$. We consider the following observation/measurement apertures
\begin{eqnarray*}
&& \gamma_1^o=\{(\cos \phi, \sin \phi) | \; \phi \in [0,2\pi]\}, \\
&& \gamma_2^o=\{(\cos \phi, \sin \phi) | \;\phi \in [0,\pi]\}, \\
&& \gamma_3^o=\{(\cos \phi, \sin \phi) | \;\phi \in [0,\pi/2]\},\\
&& \gamma_4^o=\{ (\cos \phi, \sin \phi) | \;\phi \in [0,\pi/2]\cup[\pi,3\pi/2]\}, \\
&& \gamma_5^o=\{(\cos \phi, \sin \phi) | \;\phi \in [0,\pi/4]\cup [\pi,5\pi/4]\}.
\end{eqnarray*}
The incident apertures are
\begin{eqnarray*}
&& \gamma_1^i = \{ (1,0) \}, \\
&& \gamma_2^i = \{(\cos\theta,\sin \theta) |\; \theta=\{0,\pi/8,\pi/4,3\pi/8,\pi/2\}\}.
\end{eqnarray*}
%

%

\subsection{Different Boundary Parameterizations}
We first check the reconstructions of different boundary parameterizations $q_a$, $q_b$ and $q_c$ using the kite.
The limited aperture data is $u^{\infty}(\hat{x}, d), (\hat{x}, d) \in \gamma_1^o \times \gamma_1^i$, i.e., far field pattern of all direction due to one incident plane wave.
We set $\beta=0.0001$ and $\sigma = 10\%$. The number of iteration is set to be $K=10, 000$. The location $z_0=(0.2, 0)$ is obtained by the ESM.
In Figure~\ref{qs}, we plot the reconstructions of the boundary of three different parameterizations $q_a$, $q_b$, and $q_c$ (left column) and
Markov chains for $a_0$. The results show that $q_b$ performs better. In the following examples, we shall use $q_b$.
\begin{figure}[ht]
\begin{center}
\begin{tabular}{cc}
\resizebox{0.52\textwidth}{!}{\includegraphics{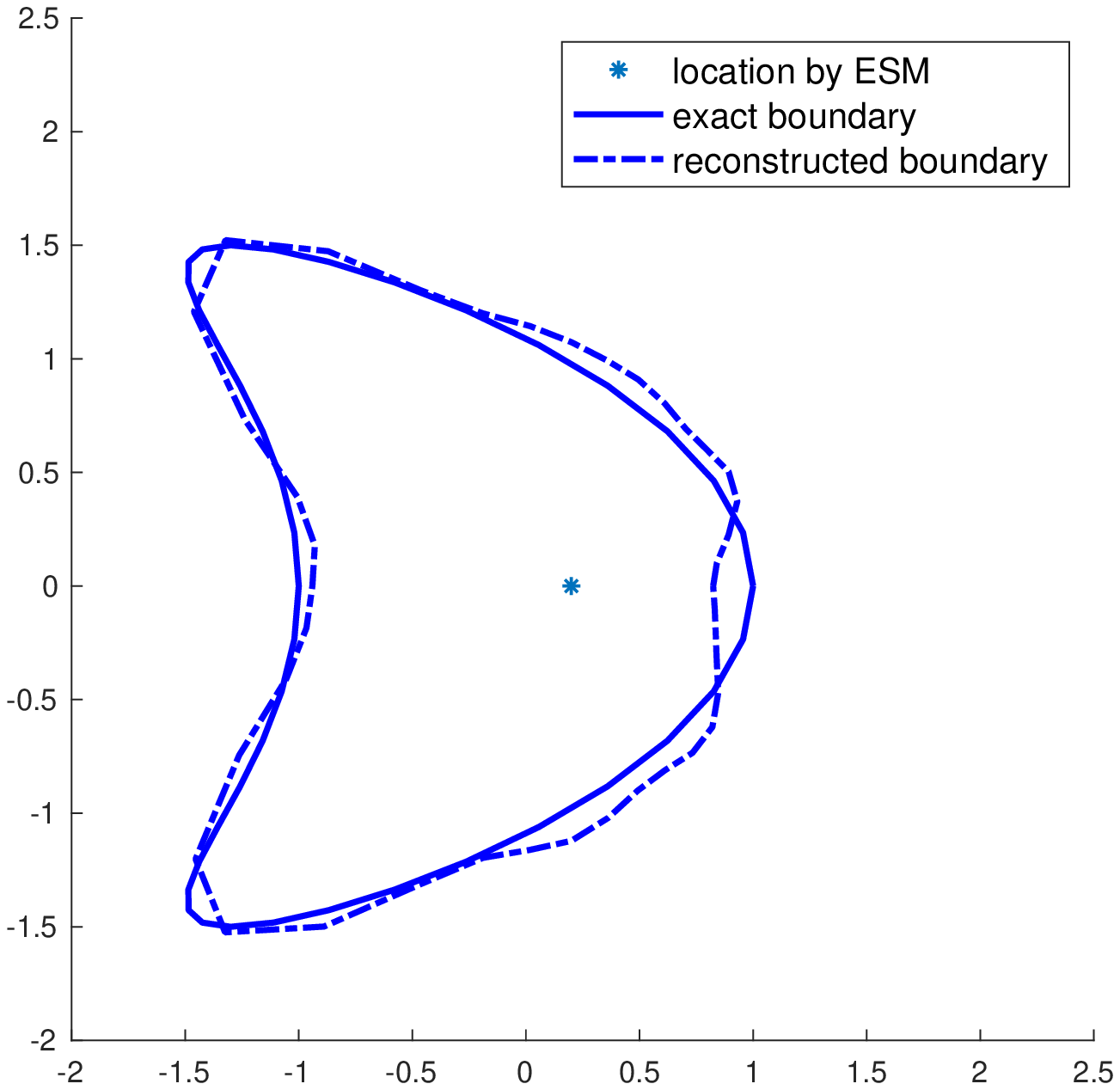}}&
\resizebox{0.52\textwidth}{!}{\includegraphics{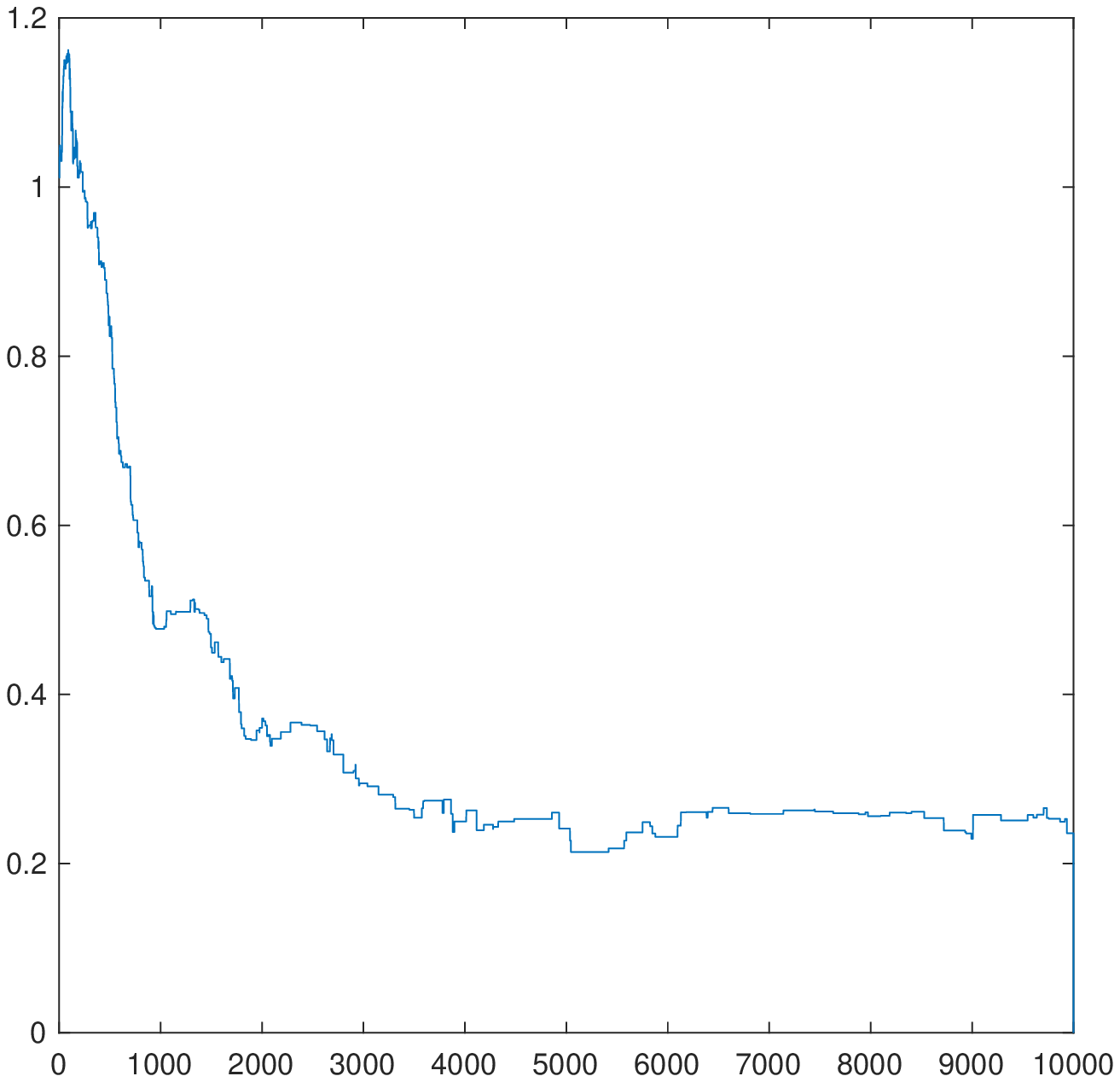}}\\
\resizebox{0.52\textwidth}{!}{\includegraphics{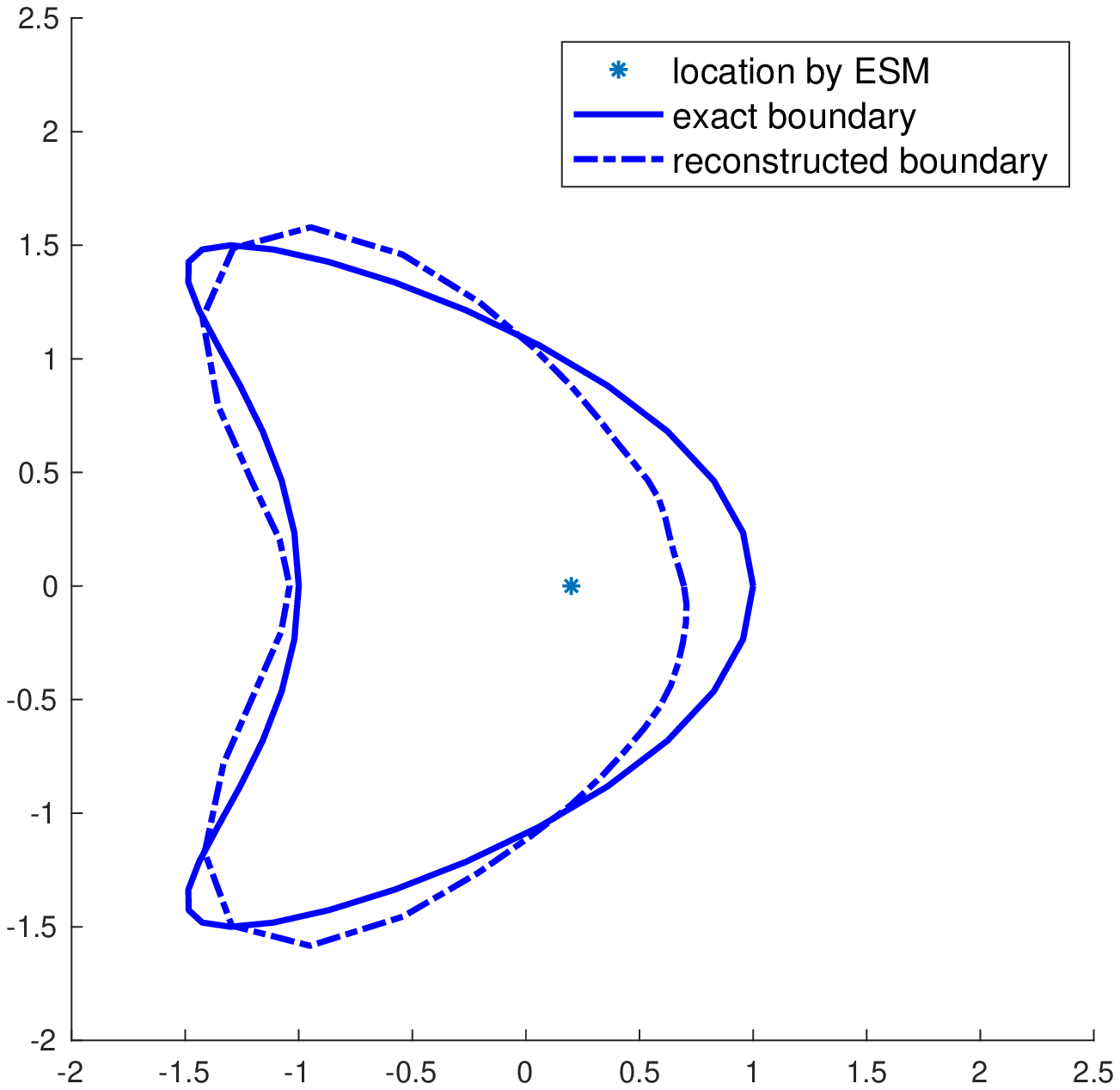}}&
\resizebox{0.52\textwidth}{!}{\includegraphics{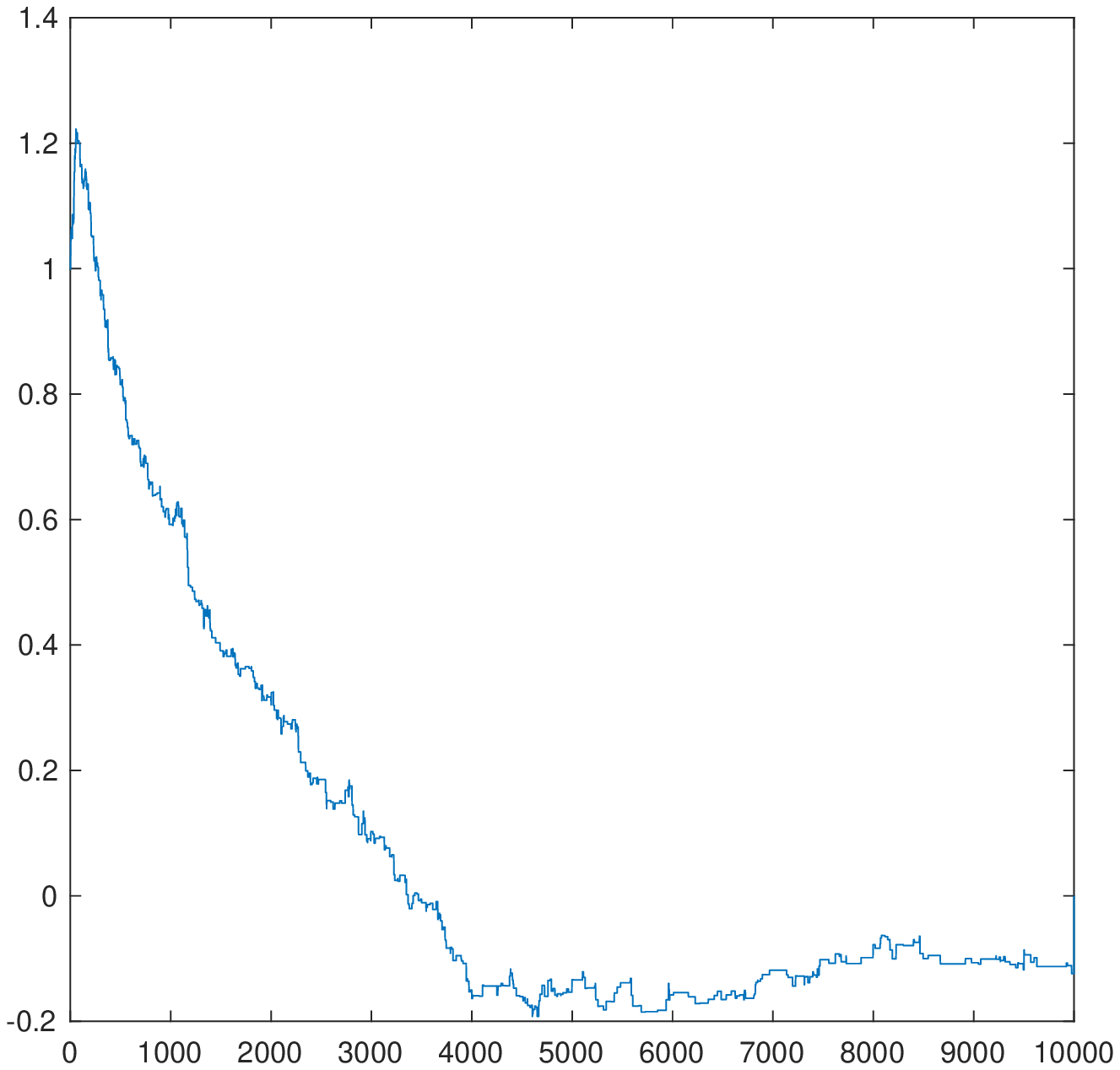}}\\
\resizebox{0.52\textwidth}{!}{\includegraphics{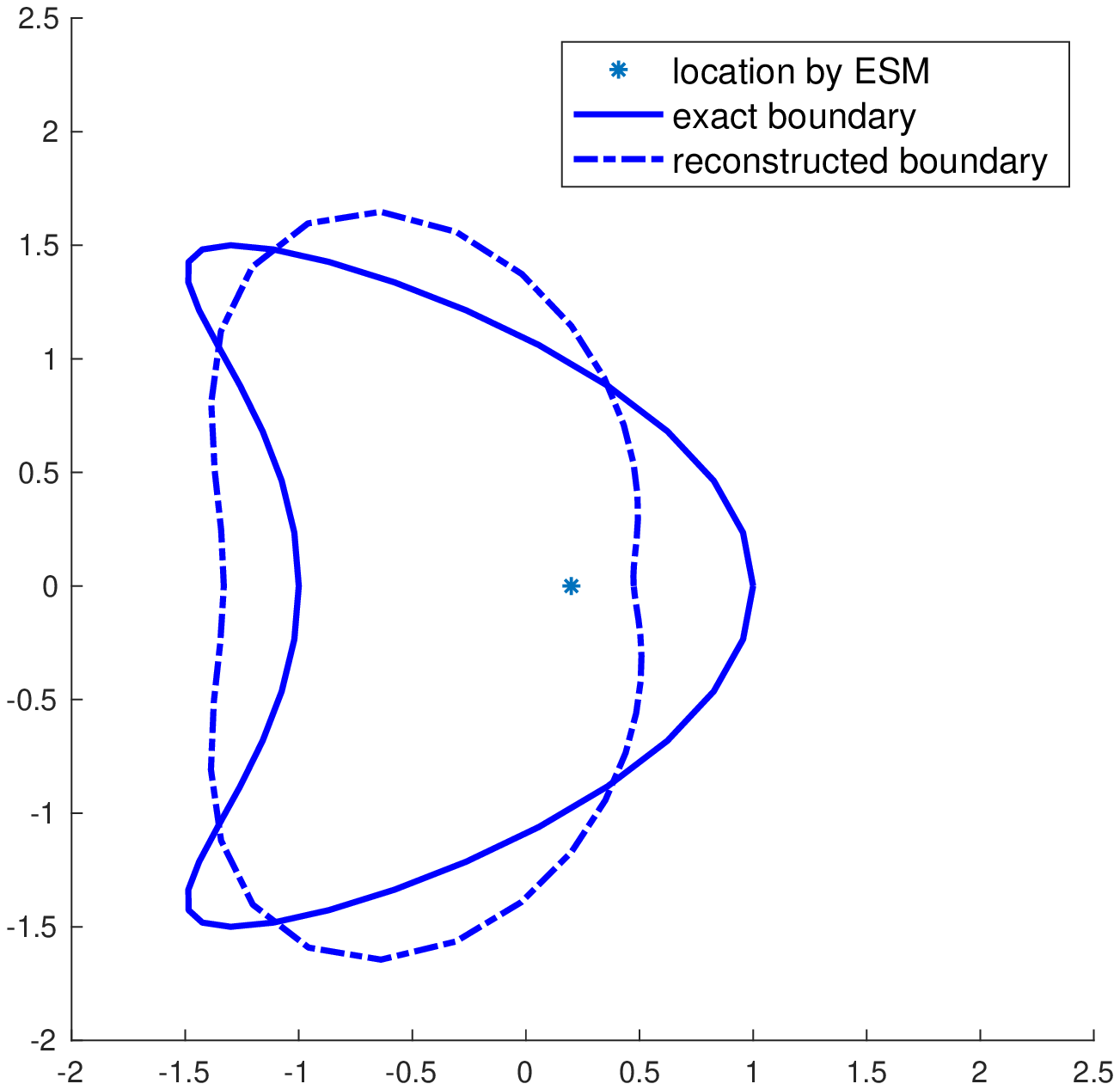}}&
\resizebox{0.52\textwidth}{!}{\includegraphics{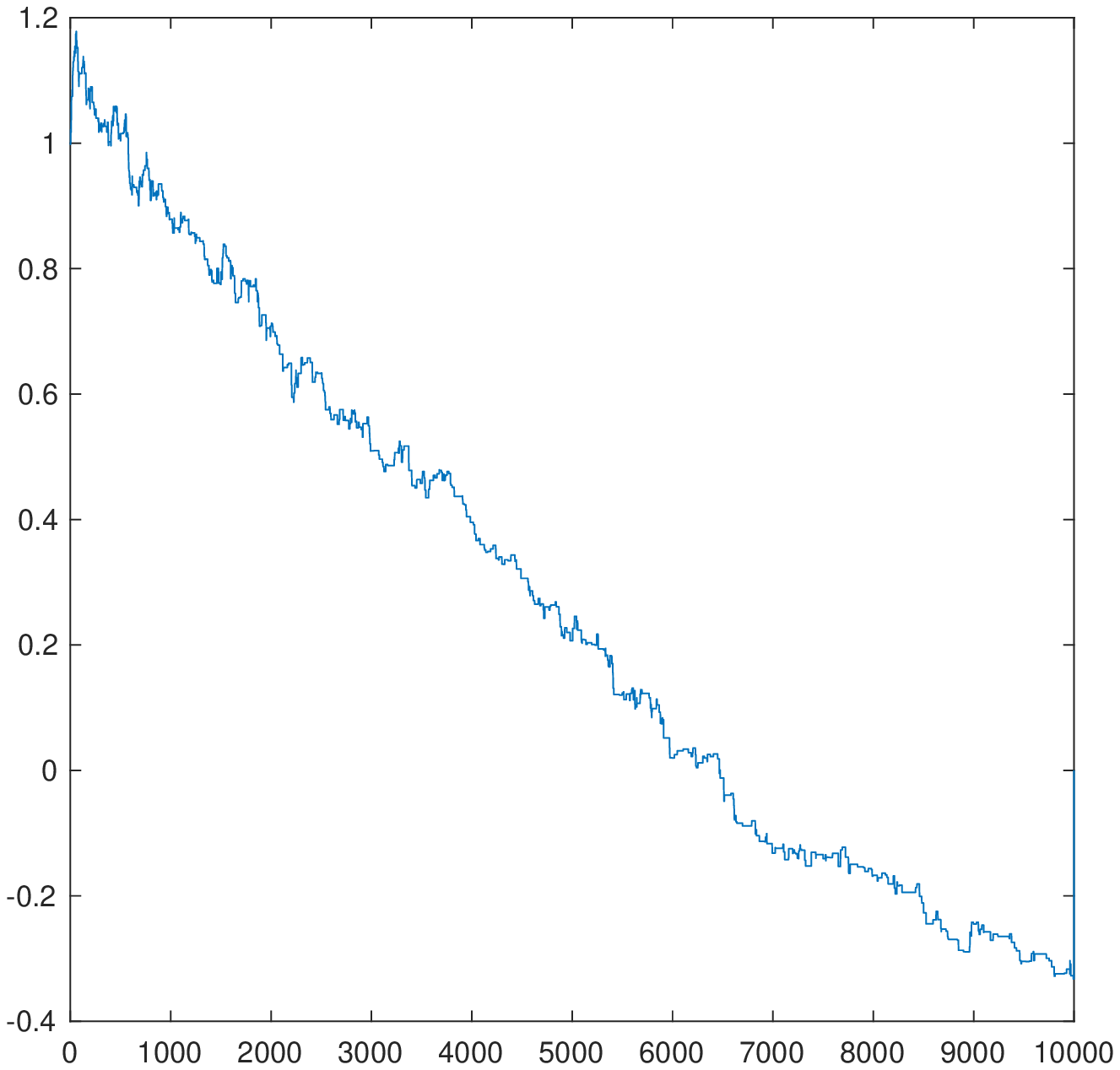}}
\end{tabular}
\end{center}
\caption{Reconstructions of the kite of different boundary parameterizations using $u^{\infty}(\hat{x}, d), (\hat{x}, d) \in \gamma_1^o \times \gamma_1^i$. 
Left column: boundary reconstructions. Right column: Markov chains for $a_0$. Top to bottom: $q_a$, $a_b$ and $q_c$.}
\label{qs}
\end{figure}

\subsection{Different Parameters}
Different $\beta$ in \eqref{kernelq} and $\eta$ in \eqref{9a} lead to different acceptance rates.
Table \eqref{betaeta} shows the acceptance rates for $K=10,000$. The results show that smaller $\eta$'s lead to lower acceptance rates
while smaller $\beta$'s lead to higher acceptance rates.
\begin{table}[h!]
\caption{Acceptance rates}
\label{betaeta}
\centering
\begin{tabular}{l|c|c|c|c}
\hline
 & $\beta=10^{-2}$&$\beta=10^{-3}$&  $\beta=10^{-4}$& $\beta=10^{-5}$\\ \hline
$\eta=0.1$  &0.0146 &0.1351&0.5093 & 0.8201\\ \hline
$\eta=0.05$  & 0.0085 &0.0476&0.2649 &0.6543\\ \hline
 $\eta=0.01$  & 0.0058 &0.0192&0.0597 & 0.2204\\
\hline	
\end{tabular}
\end{table}

\subsection{Different Data Apertures}
Now we show some numerical results for different limited aperture data.
We first consider the case $\gamma^i = \gamma^i_1$, i.e., there is only one incident wave.
We take three observation apertures $\gamma^o_1$, $\gamma^o_2$ and $\gamma^o_3$.
In Fig.~\ref{Fig1}, we show the reconstructions of the boundary for the kite when
\begin{eqnarray*}
&& \quad \gamma^o \times \gamma^i = \gamma_2^o \times \gamma^i_1,\\
&& \quad \gamma^o \times \gamma^i = \gamma_3^o \times \gamma^i_1,
\end{eqnarray*}
respectively. The $*$'s represent the locations reconstructed by the ESM 
\[
(-0.1, 0.1),\quad (0,0.3)
\]
for cases (1), (2), and (3), respectively.
The dotted line is the reconstructed boundary using the CM of the posterior probability distribution for $q$.
The solid line is the exact boundary. 
\begin{figure}[ht]
\begin{center}
\begin{tabular}{cc}
\resizebox{0.52\textwidth}{!}{\includegraphics{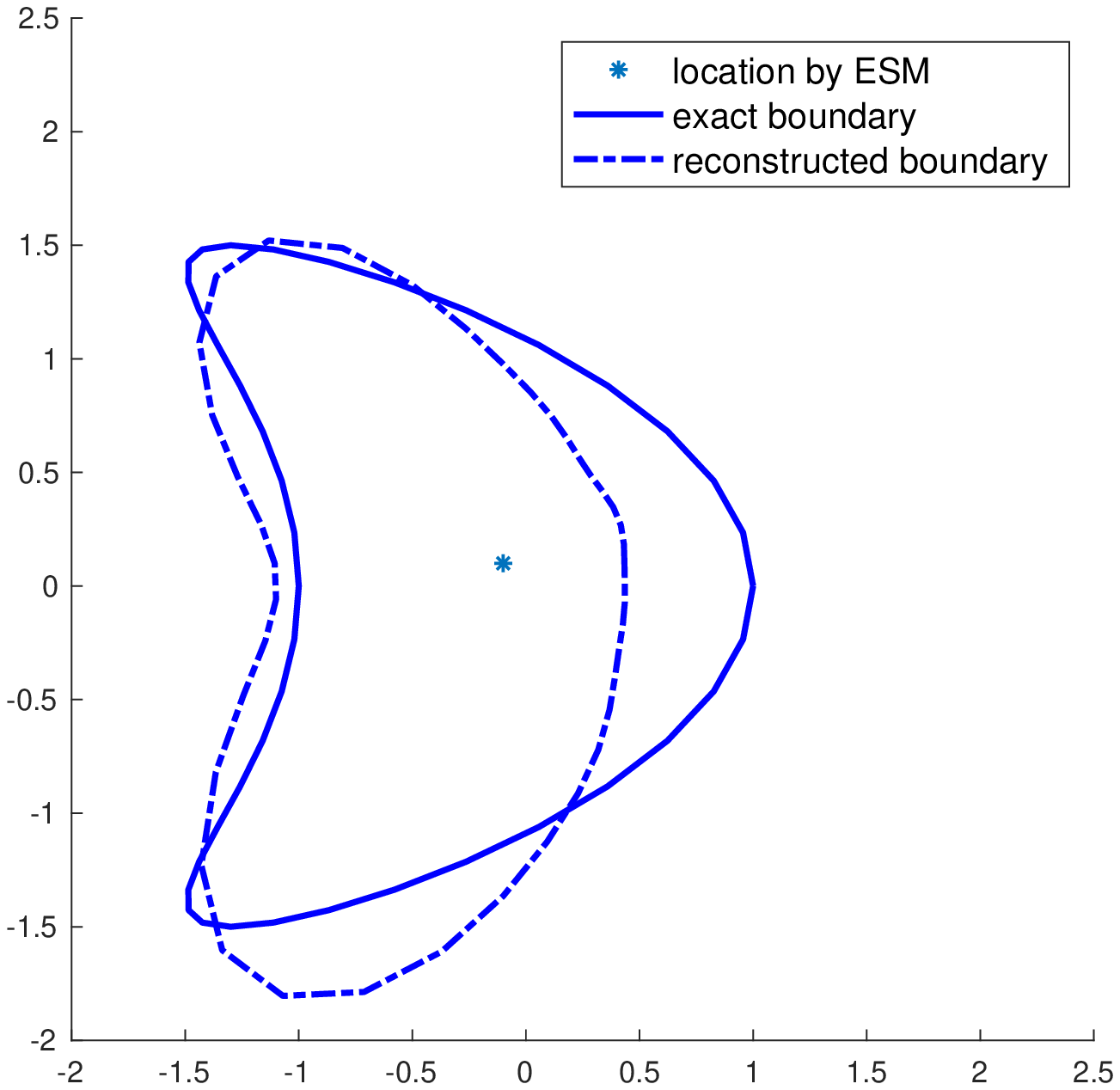}}&
\resizebox{0.52\textwidth}{!}{\includegraphics{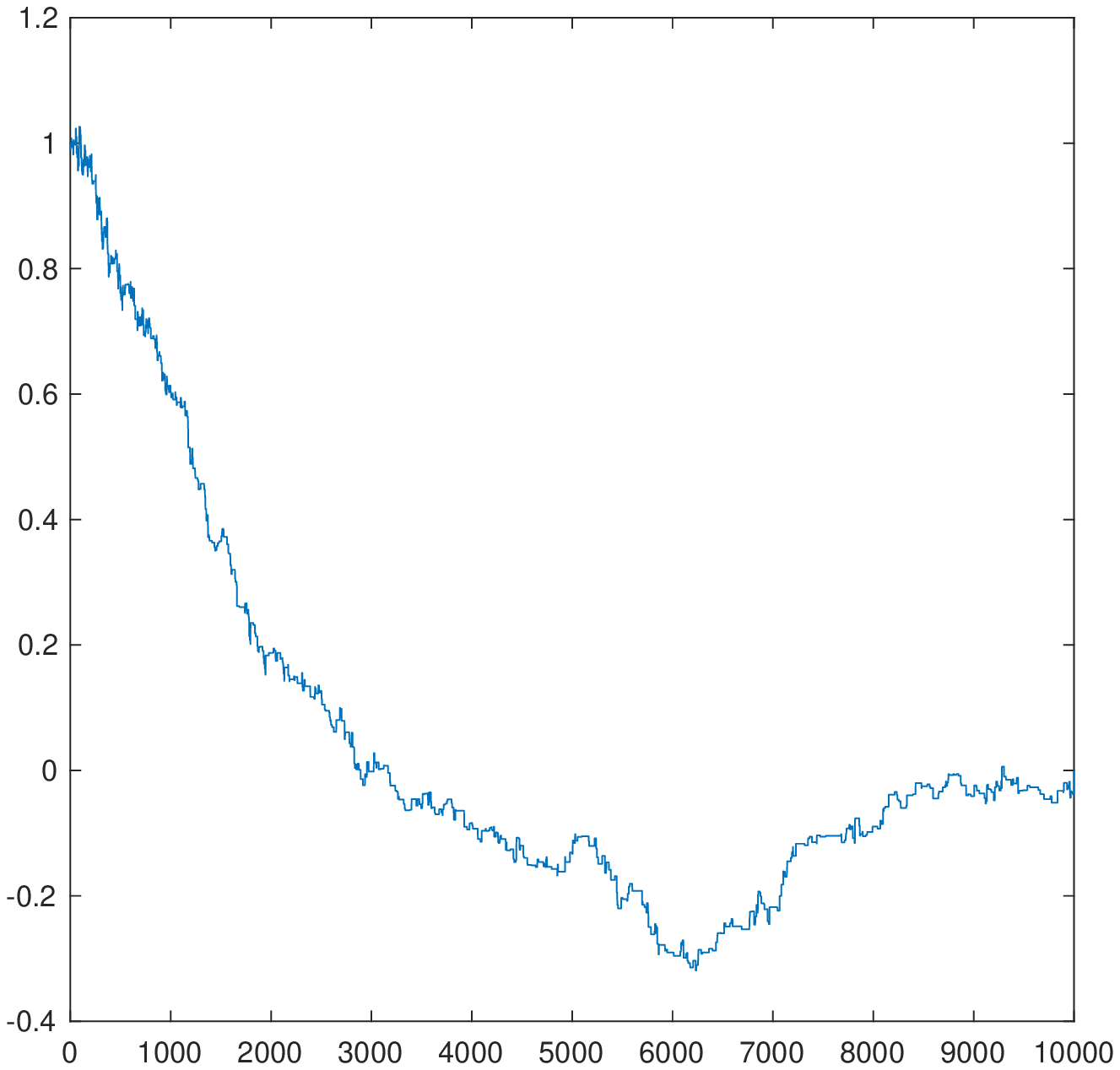}}\\
\resizebox{0.52\textwidth}{!}{\includegraphics{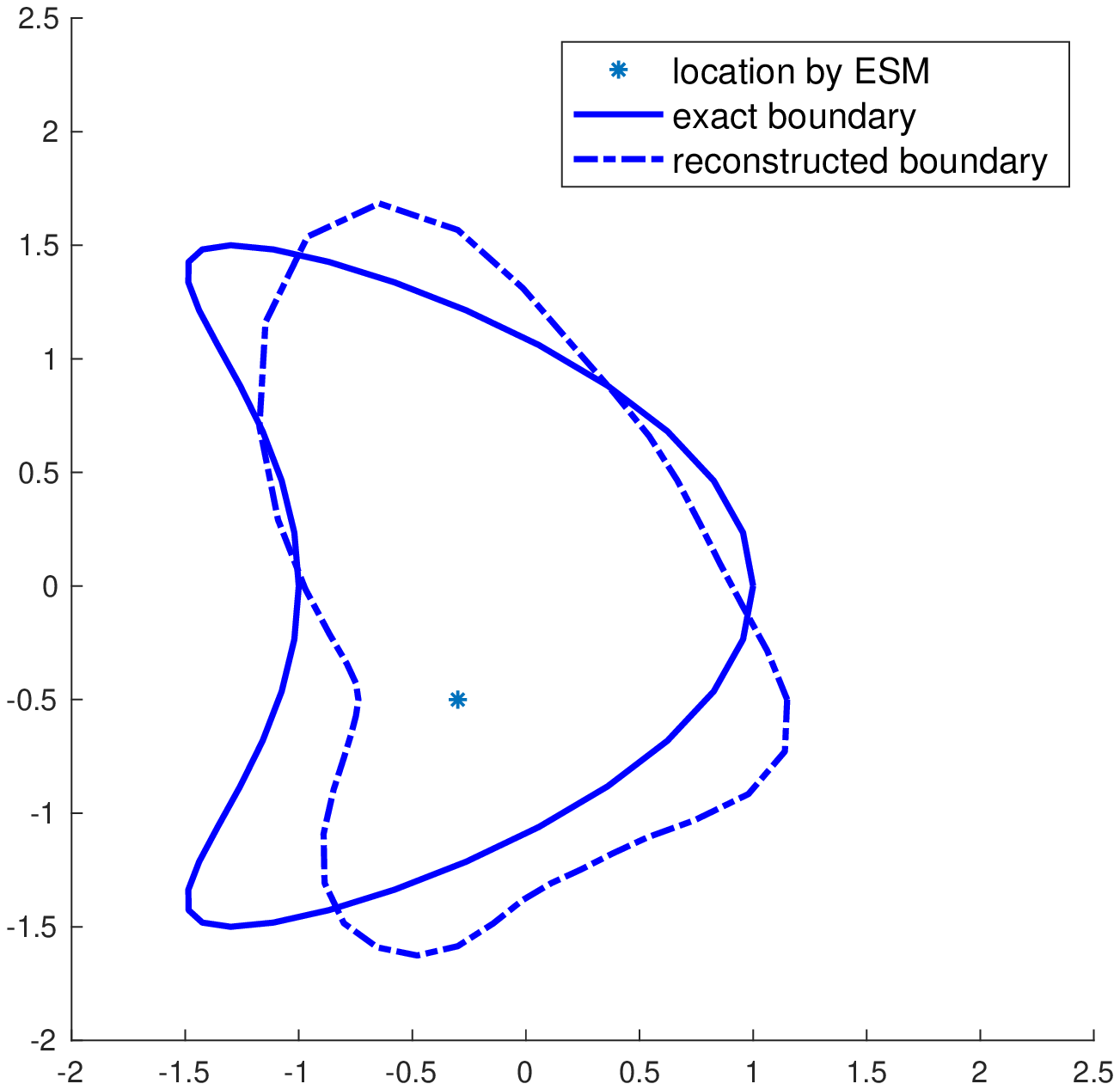}}&
\resizebox{0.52\textwidth}{!}{\includegraphics{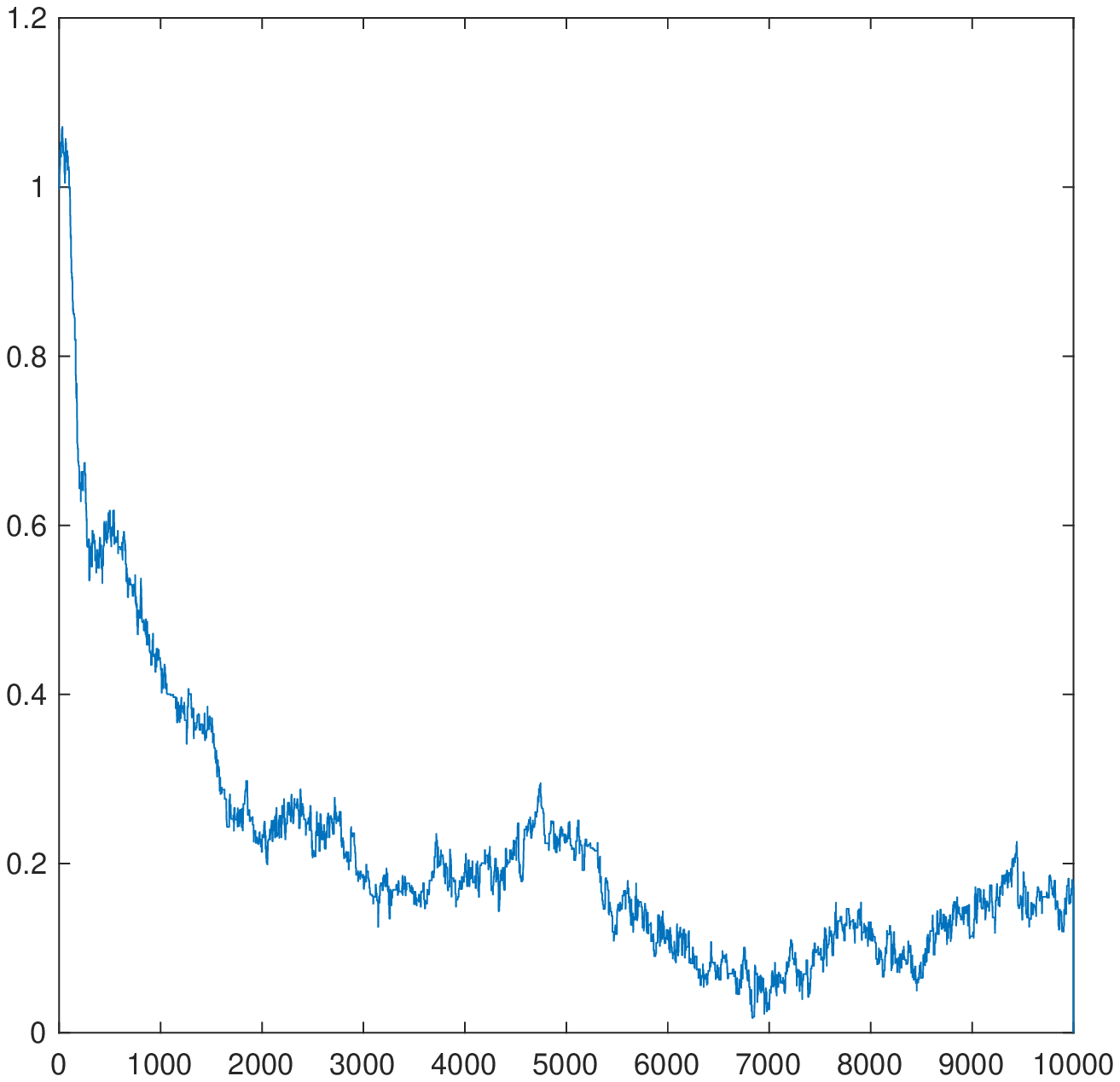}}
\end{tabular}
\end{center}
\caption{Reconstructions of the kite and  Markov chains for $a_0$.
Top: $\gamma^o \times \gamma^i = \gamma_2^o \times \gamma^i_1$. Bottom: $\gamma^o \times \gamma^i = \gamma_3^o \times \gamma^i_1$.}
\label{Fig1}
\end{figure}

Next we take $\gamma^i = \gamma^i_2$, i.e., multiple incident waves.
We take three observation apertures $\gamma^o_3$, $\gamma^o_4$ and $\gamma^o_5$.
In Fig.~\ref{Fig2}, we show the reconstructions of the kite for 
\begin{eqnarray*}
&&  \gamma^o \times \gamma^i = \gamma_3^o \times \gamma^i_2,\\
&&  \gamma^o \times \gamma^i = \gamma_4^o \times \gamma^i_2,\\
&&  \gamma^o \times \gamma^i = \gamma_5^o \times \gamma^i_2,
\end{eqnarray*}
respectively. The corresponding locations by ESM are
\[
(-0.3, -0.5), \quad (-0.4, 0.1),\quad (0.2,-0.3),
\]
respectively.
\begin{figure}[ht]
\begin{center}
\begin{tabular}{cc}
\resizebox{0.52\textwidth}{!}{\includegraphics{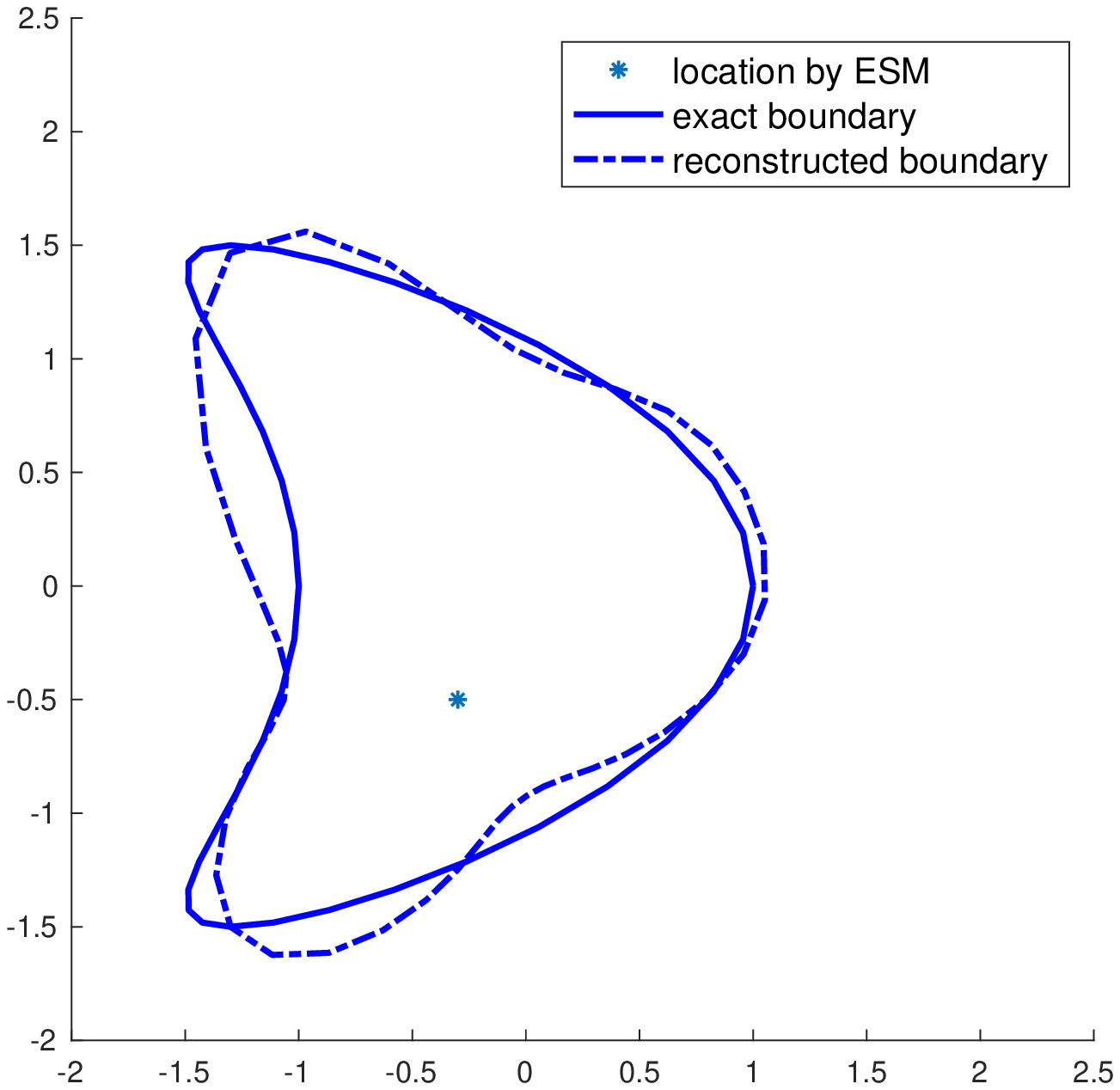}}&
\resizebox{0.52\textwidth}{!}{\includegraphics{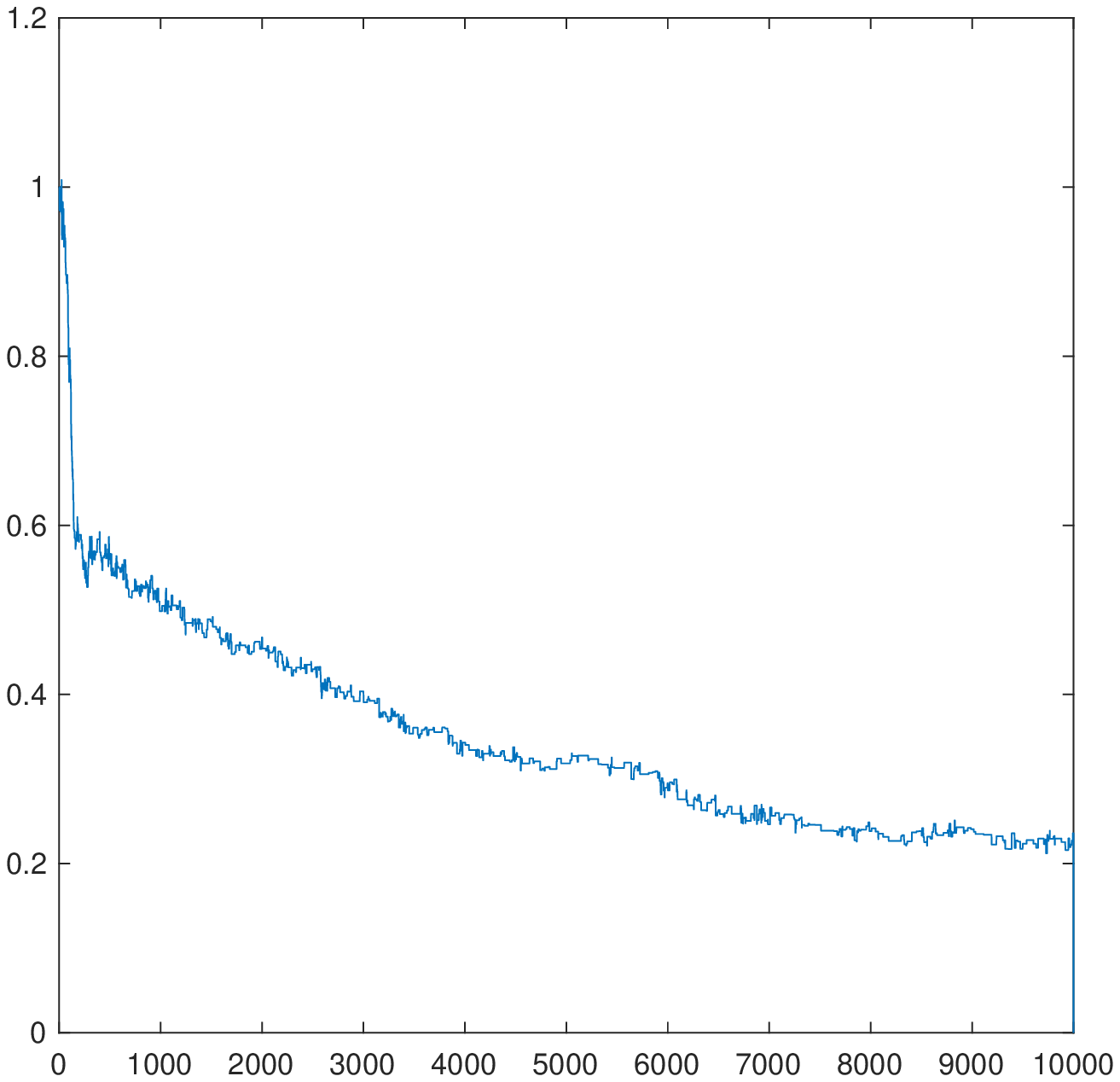}}\\
\resizebox{0.52\textwidth}{!}{\includegraphics{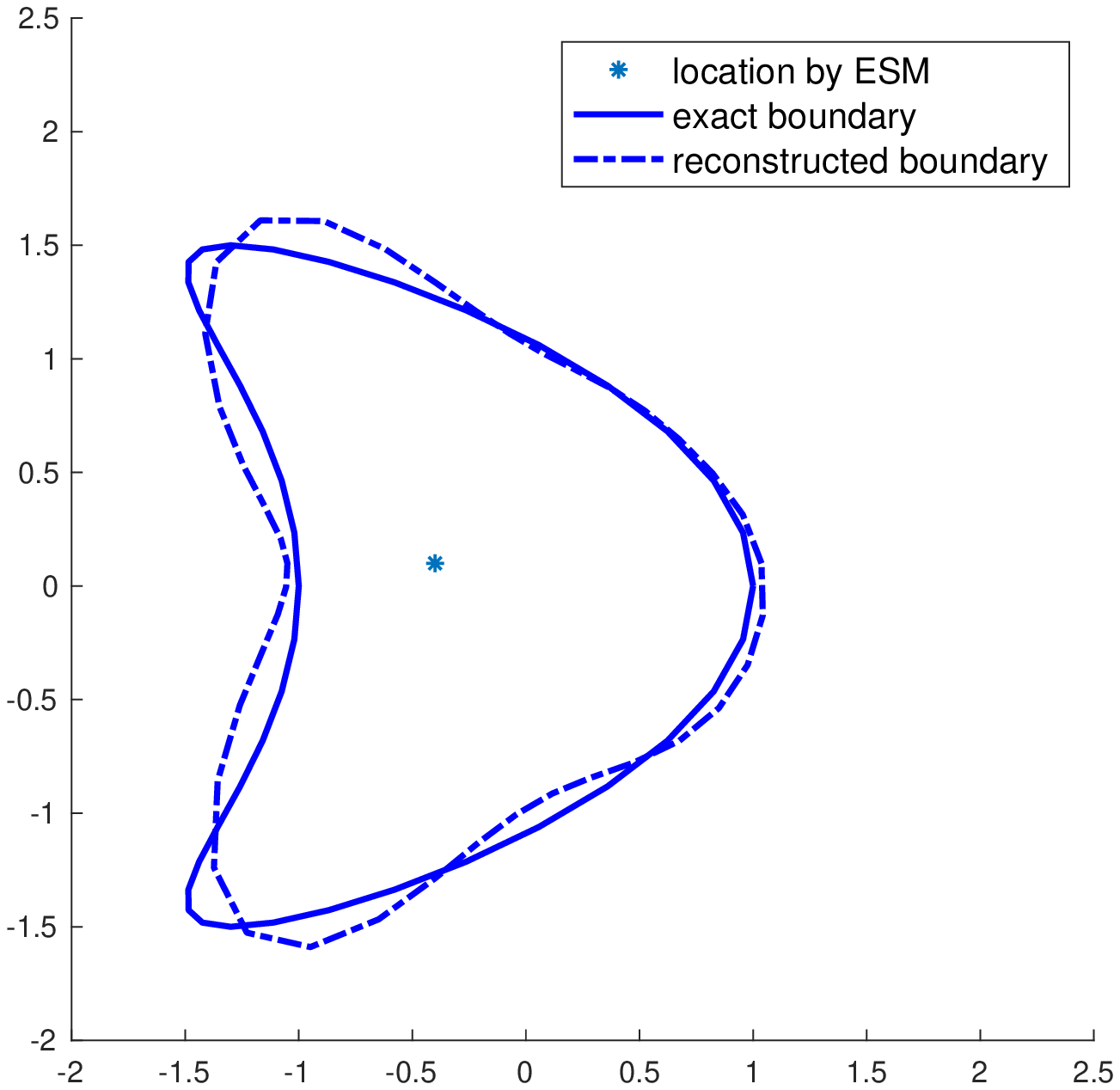}}&
\resizebox{0.52\textwidth}{!}{\includegraphics{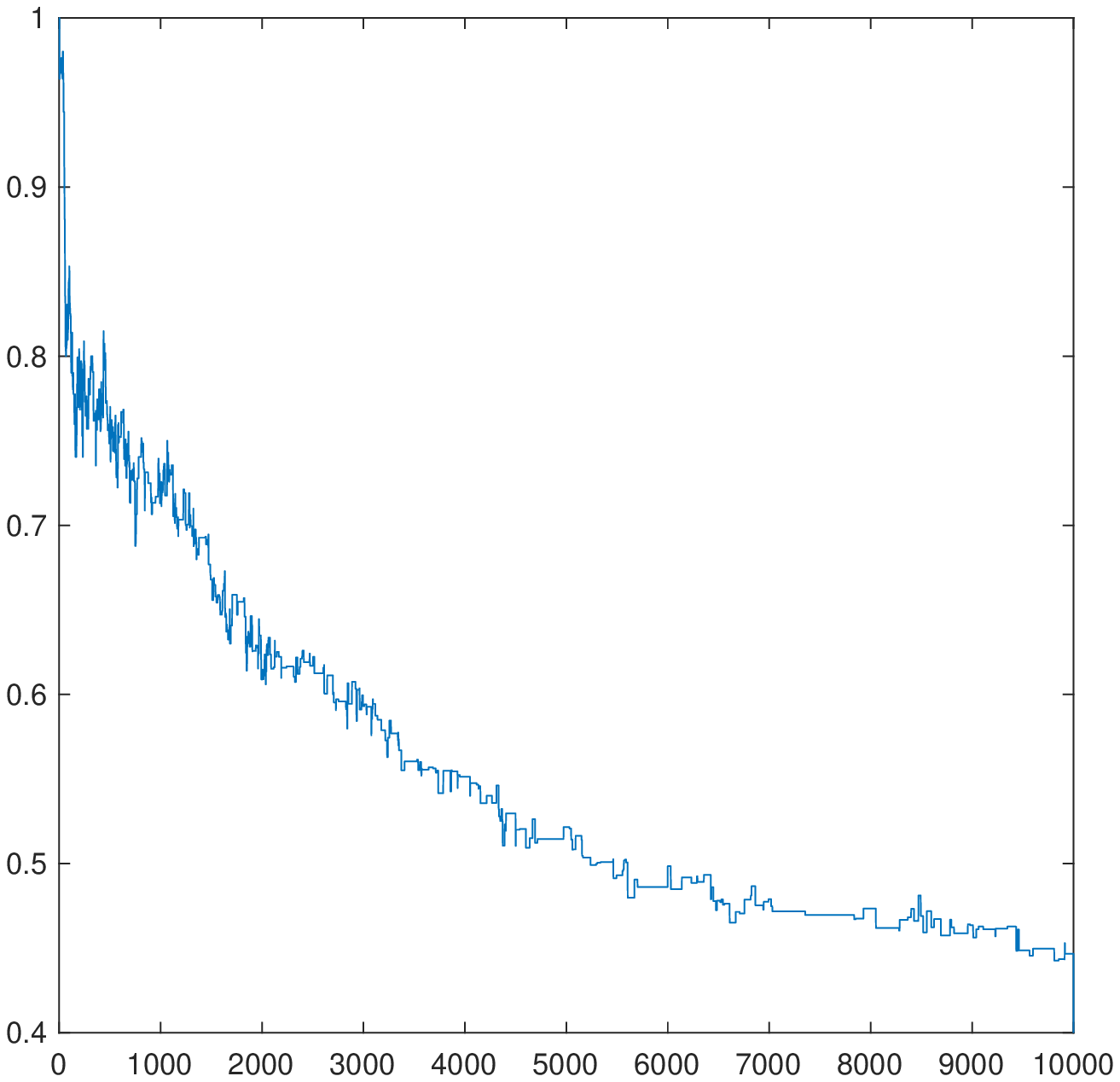}}\\
\resizebox{0.52\textwidth}{!}{\includegraphics{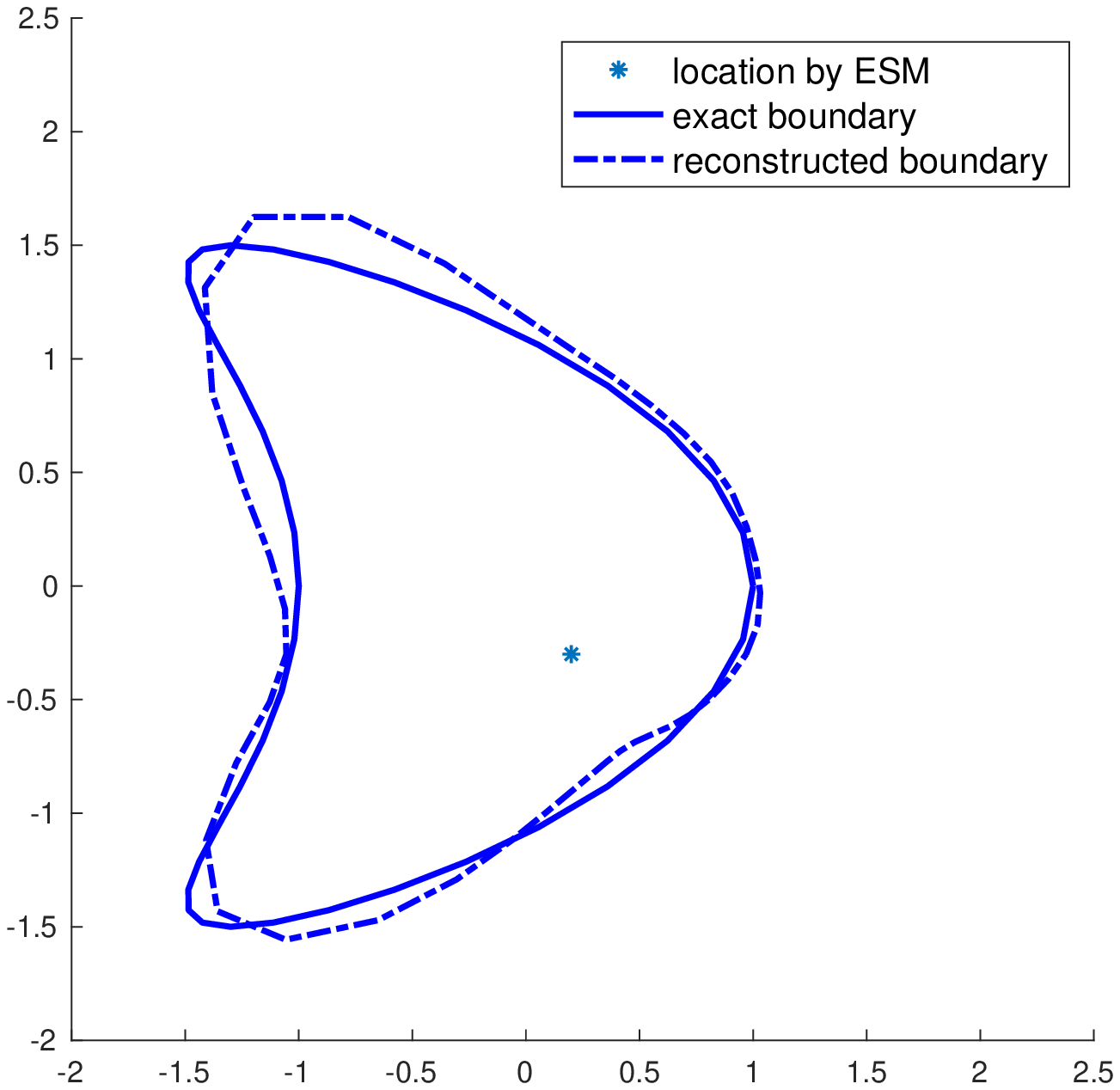}}&
\resizebox{0.52\textwidth}{!}{\includegraphics{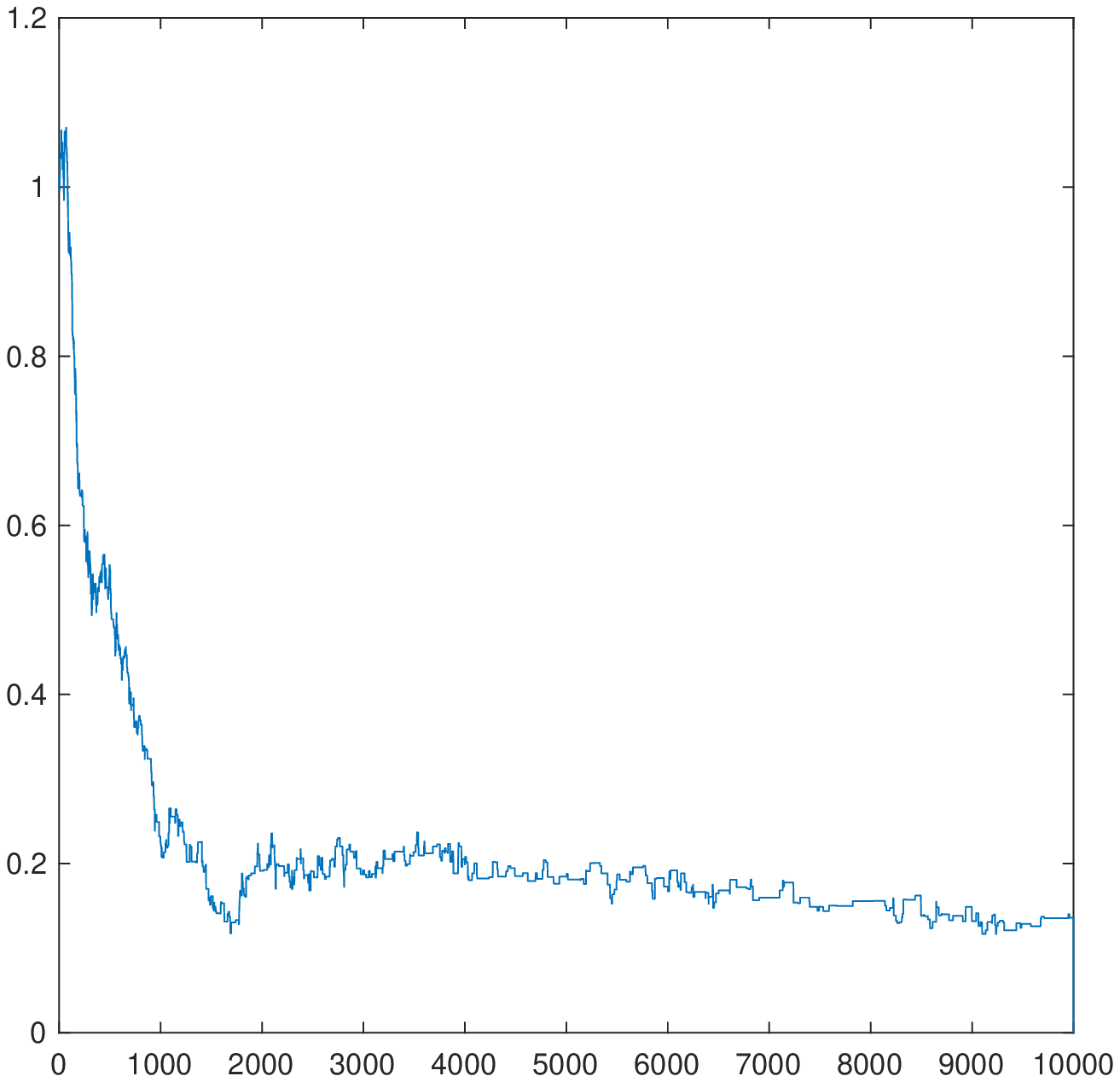}}
\end{tabular}
\end{center}
\caption{Reconstructions of the kite using multiple incident waves and Markov chains for $a_0$. Top left: $\gamma^o \times \gamma^i = \gamma_3^o \times \gamma^i_2$.
Top right: $\gamma^o \times \gamma^i = \gamma_4^o \times \gamma^i_2$. Bottom: $\gamma^o \times \gamma^i = \gamma_5^o \times \gamma^i_2$.}
\label{Fig2}
\end{figure}

Similar for the pear, we first consider the case when $\gamma^i = \gamma^i_1$ and 
three observation apertures $\gamma^o_1$, $\gamma^o_2$ and $\gamma^o_3$.
In Fig.~\ref{Fig3}, we show the reconstructions of the boundary for 
\begin{eqnarray*}
&&  \gamma^o \times \gamma^i = \gamma_1^o \times \gamma^i_1,\\
&&  \gamma^o \times \gamma^i = \gamma_2^o \times \gamma^i_1,\\
&&  \gamma^o \times \gamma^i = \gamma_3^o \times \gamma^i_1.
\end{eqnarray*}
The locations by ESM are
\[
(-0.1,-0.1),\quad (-0.2,0.1),\quad (0.2,-0.3),
\]
respectively.

\begin{figure}[ht]
\begin{center}
\begin{tabular}{cc}
\resizebox{0.52\textwidth}{!}{\includegraphics{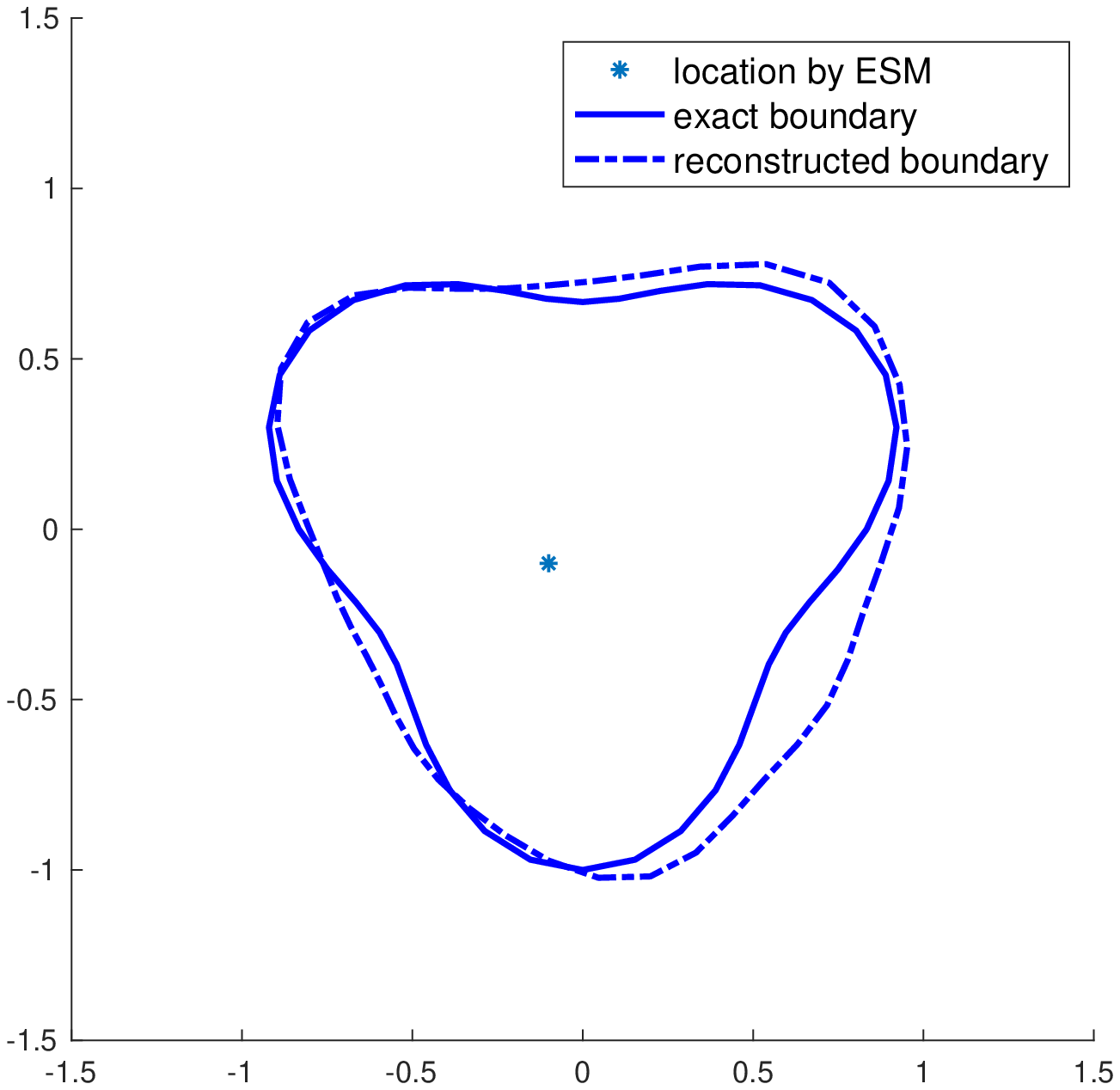}}&
\resizebox{0.52\textwidth}{!}{\includegraphics{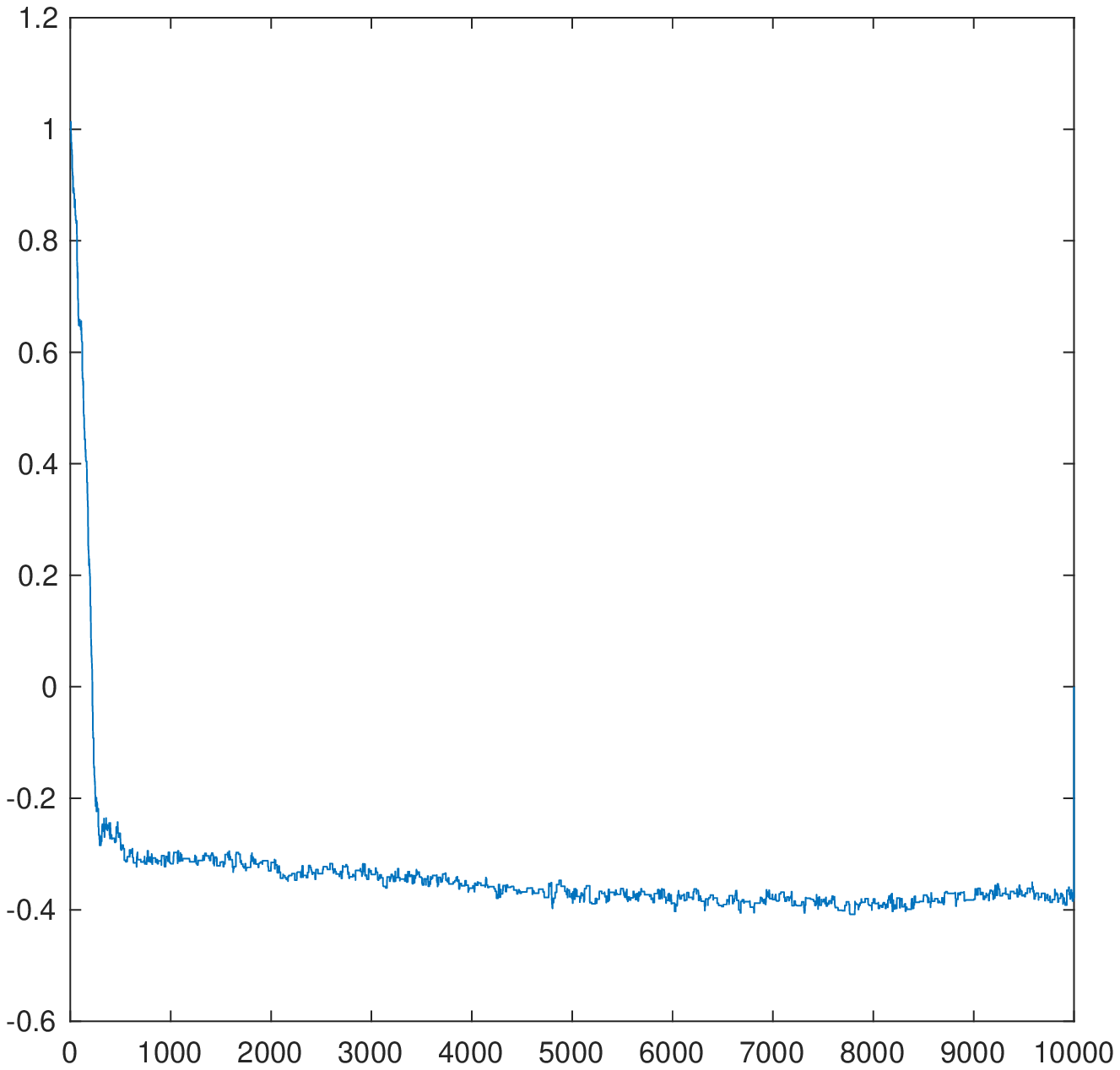}}\\
\resizebox{0.52\textwidth}{!}{\includegraphics{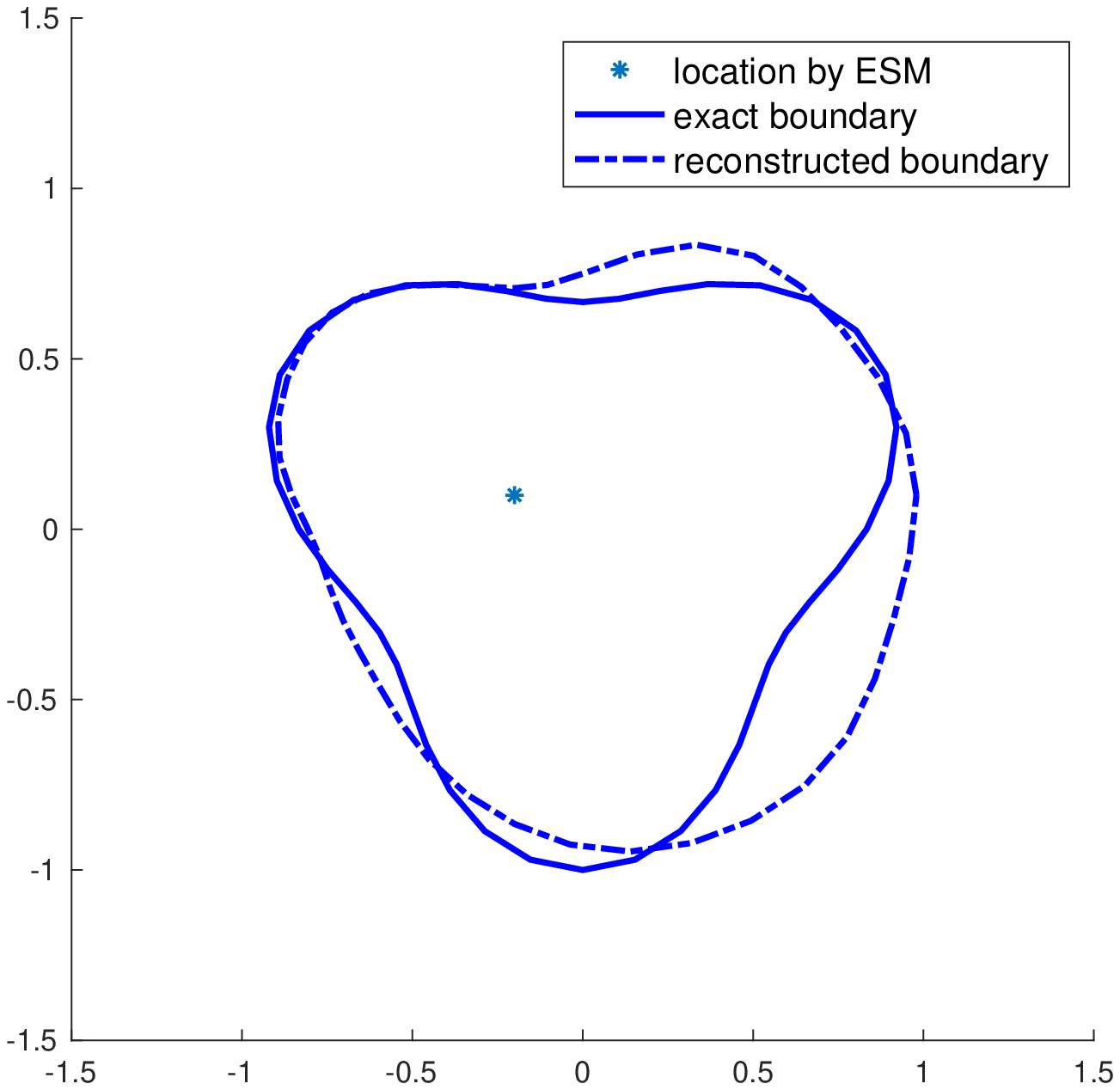}}&
\resizebox{0.52\textwidth}{!}{\includegraphics{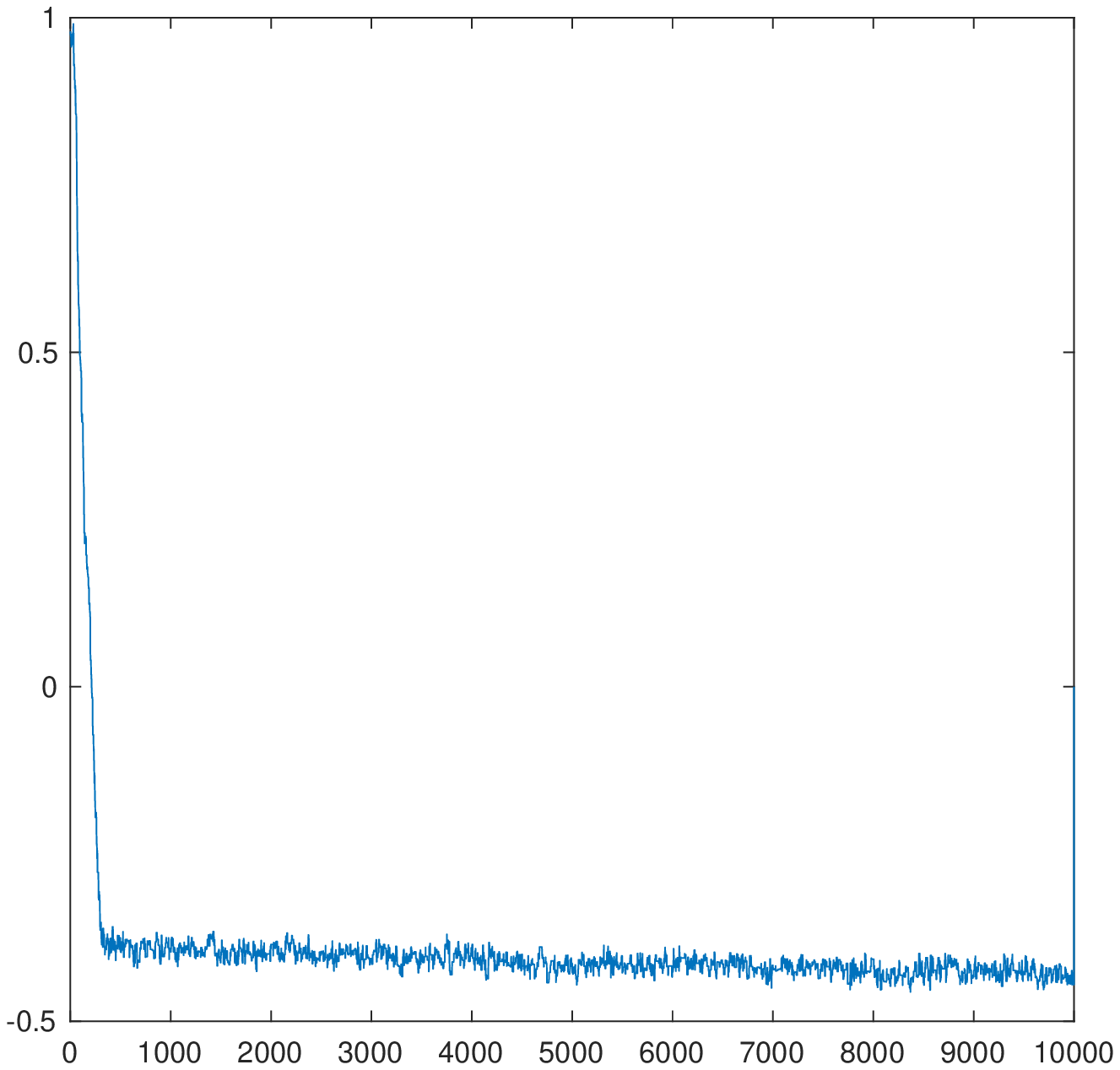}}\\
\resizebox{0.52\textwidth}{!}{\includegraphics{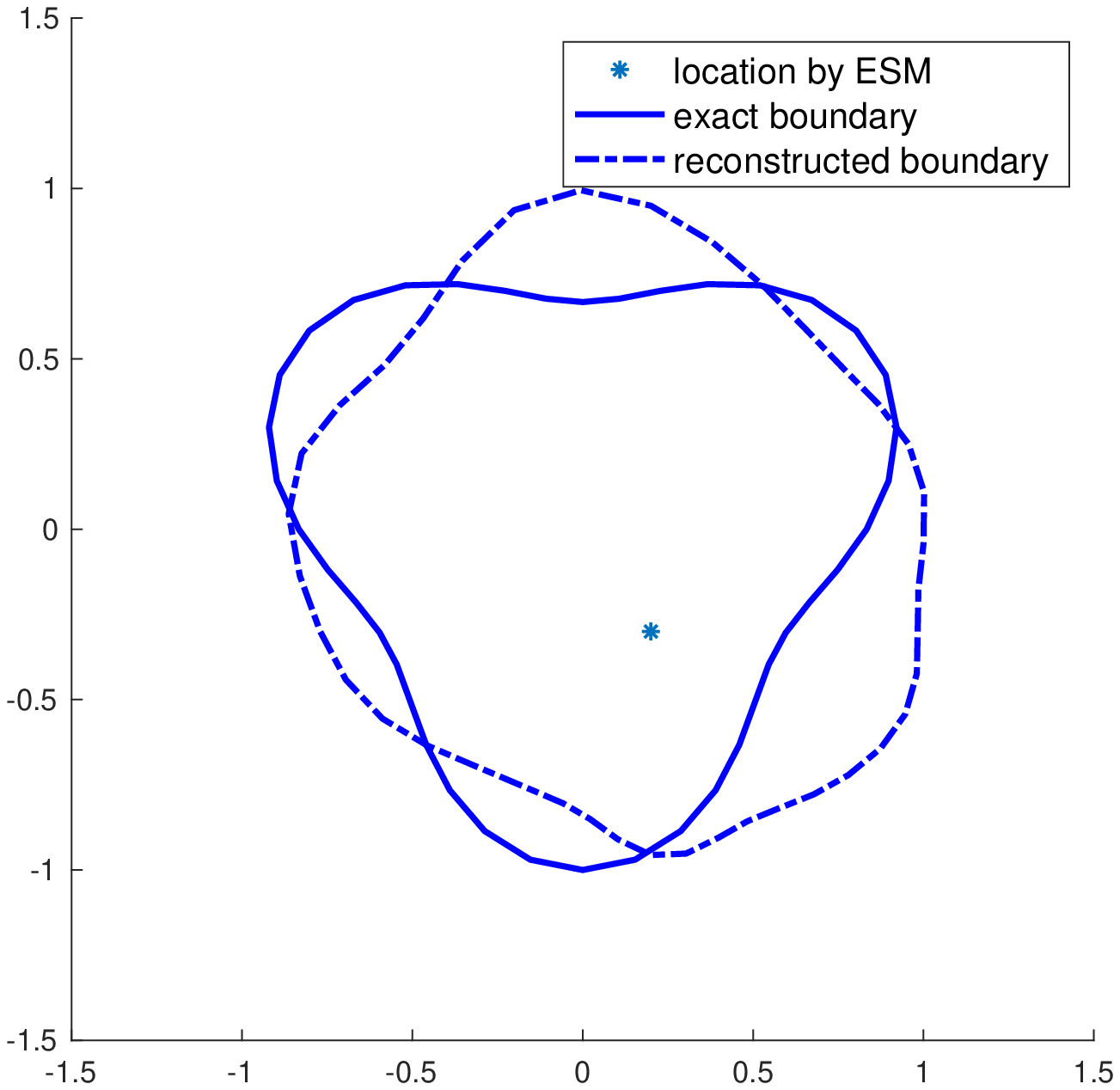}}&
\resizebox{0.52\textwidth}{!}{\includegraphics{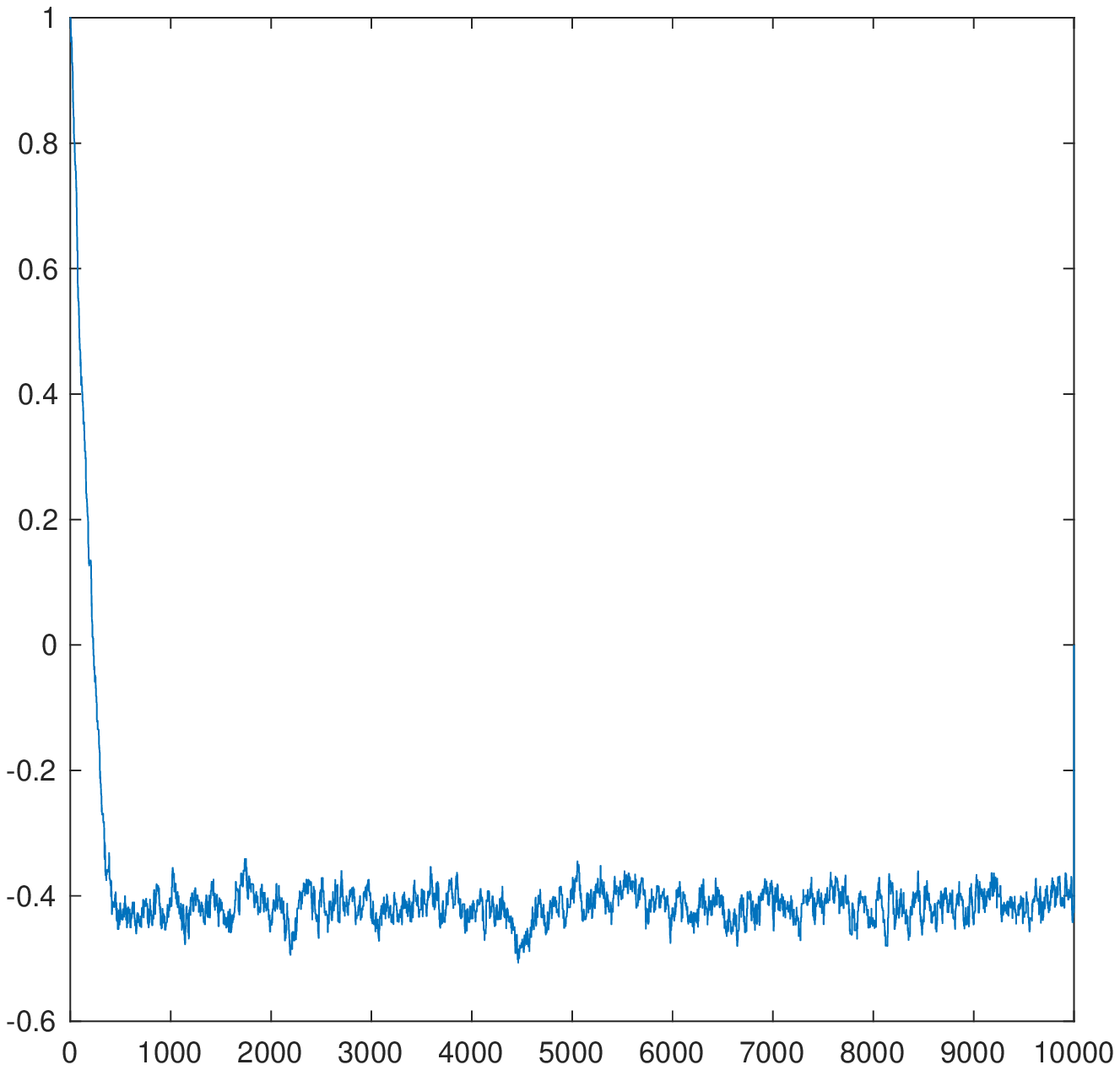}}
\end{tabular}
\end{center}
\caption{Reconstructions of the pear using one incident wave. Top left: $\gamma^o \times \gamma^i = \gamma_1^o \times \gamma^i_1$.
Top right: $\gamma^o \times \gamma^i = \gamma_2^o \times \gamma^i_1$. Bottom: $\gamma^o \times \gamma^i = \gamma_3^o \times \gamma^i_1$.}
\label{Fig3}
\end{figure}

Next we take $\gamma^i = \gamma^i_2$, i.e., multiple incident waves.
We take three observation apertures $\gamma^o_3$, $\gamma^o_4$ and $\gamma^o_5$.
In Fig.~\ref{Fig4}, we show the reconstructions of the boundary for 
\begin{eqnarray*}
&&  \gamma^o \times \gamma^i = \gamma_3^o \times \gamma^i_2,\\
&&  \gamma^o \times \gamma^i = \gamma_4^o \times \gamma^i_2,\\
&&  \gamma^o \times \gamma^i = \gamma_5^o \times \gamma^i_2,
\end{eqnarray*}
respectively. The locations by ESM are
\[
(-0.1,-0.3), \quad (-0.2,-0.2), \quad (-0.2,-0.1).
\]

\begin{figure}[ht]
\begin{center}
\begin{tabular}{cc}
\resizebox{0.52\textwidth}{!}{\includegraphics{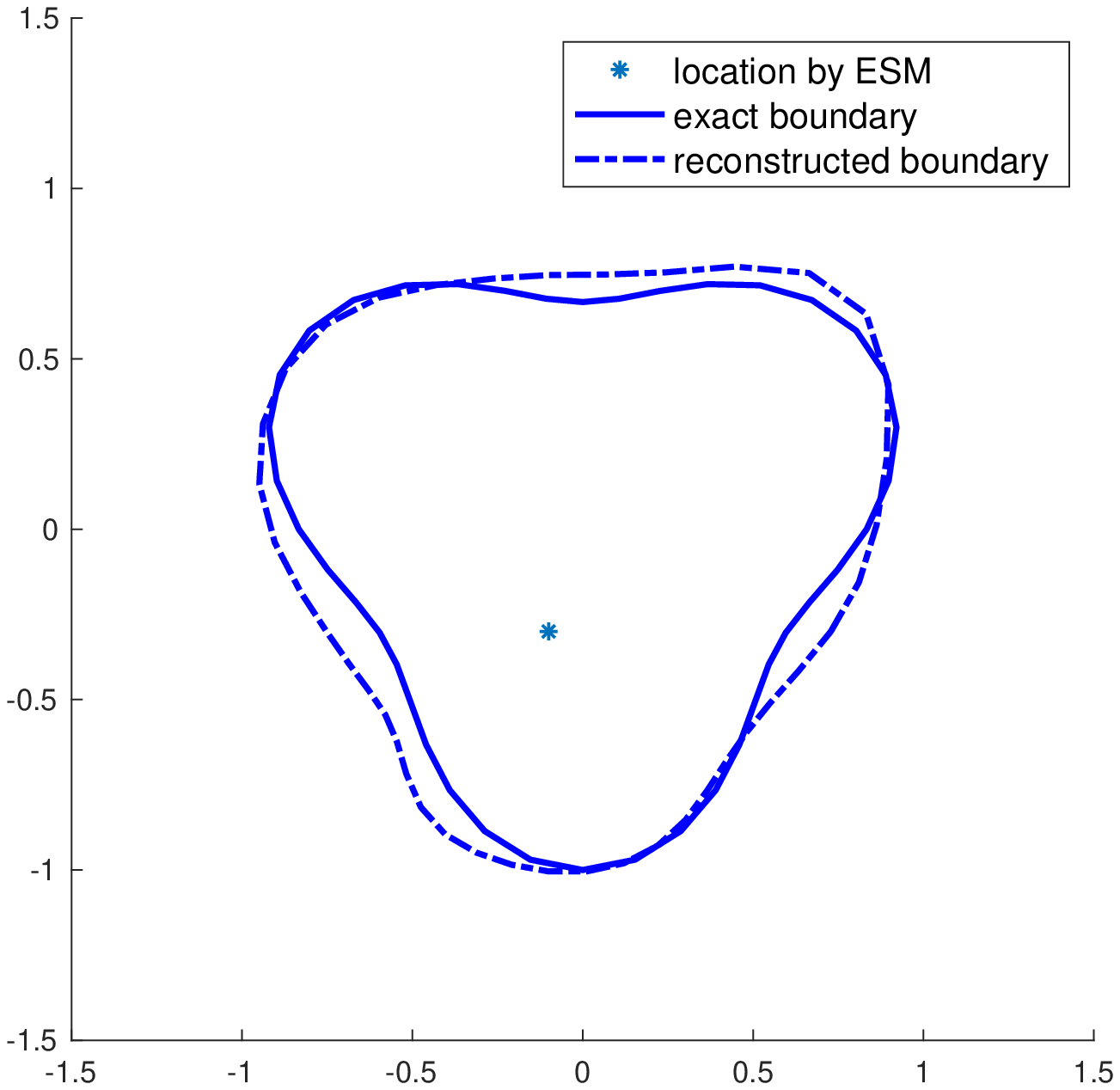}}&
\resizebox{0.52\textwidth}{!}{\includegraphics{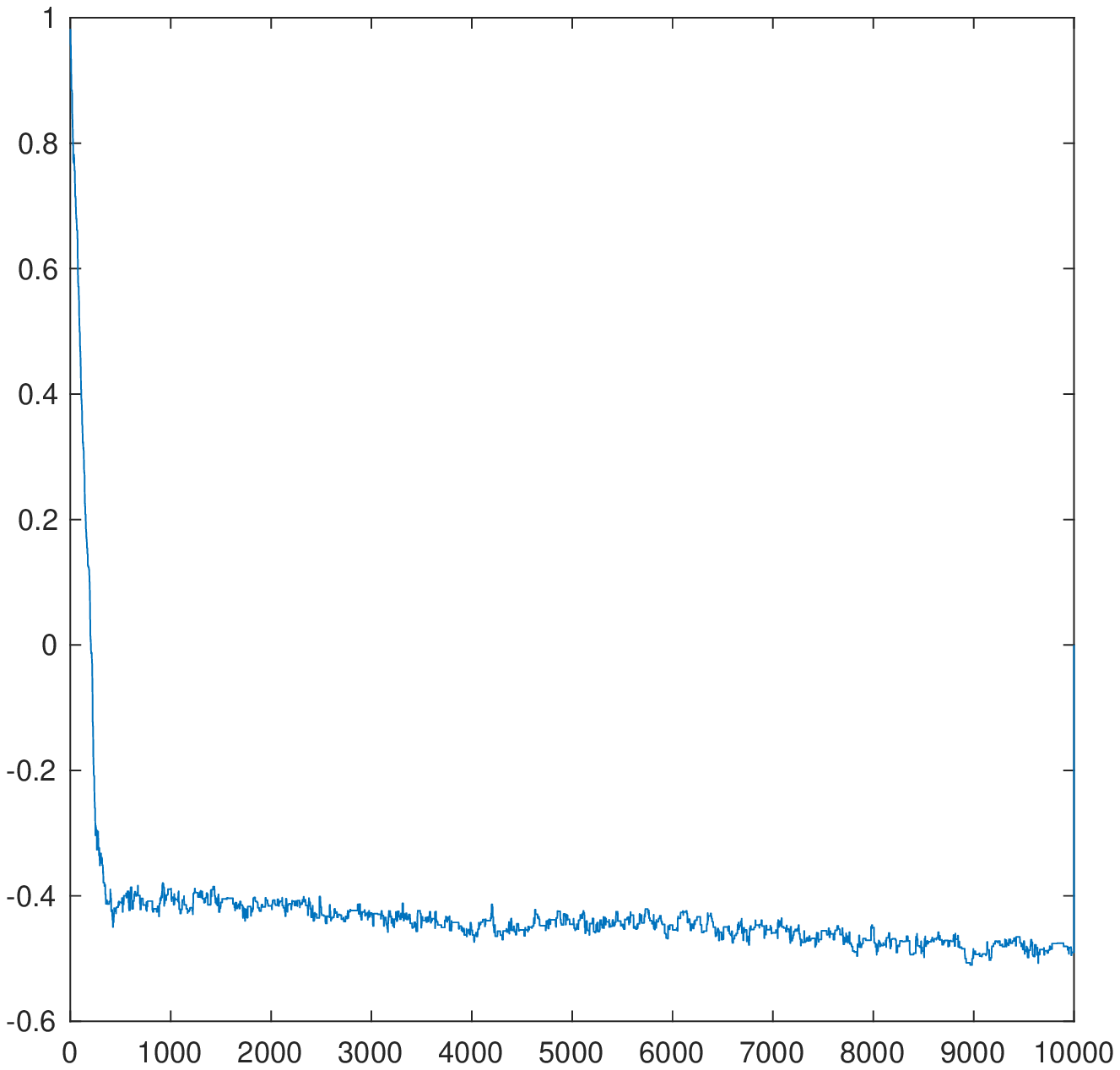}}\\
\resizebox{0.52\textwidth}{!}{\includegraphics{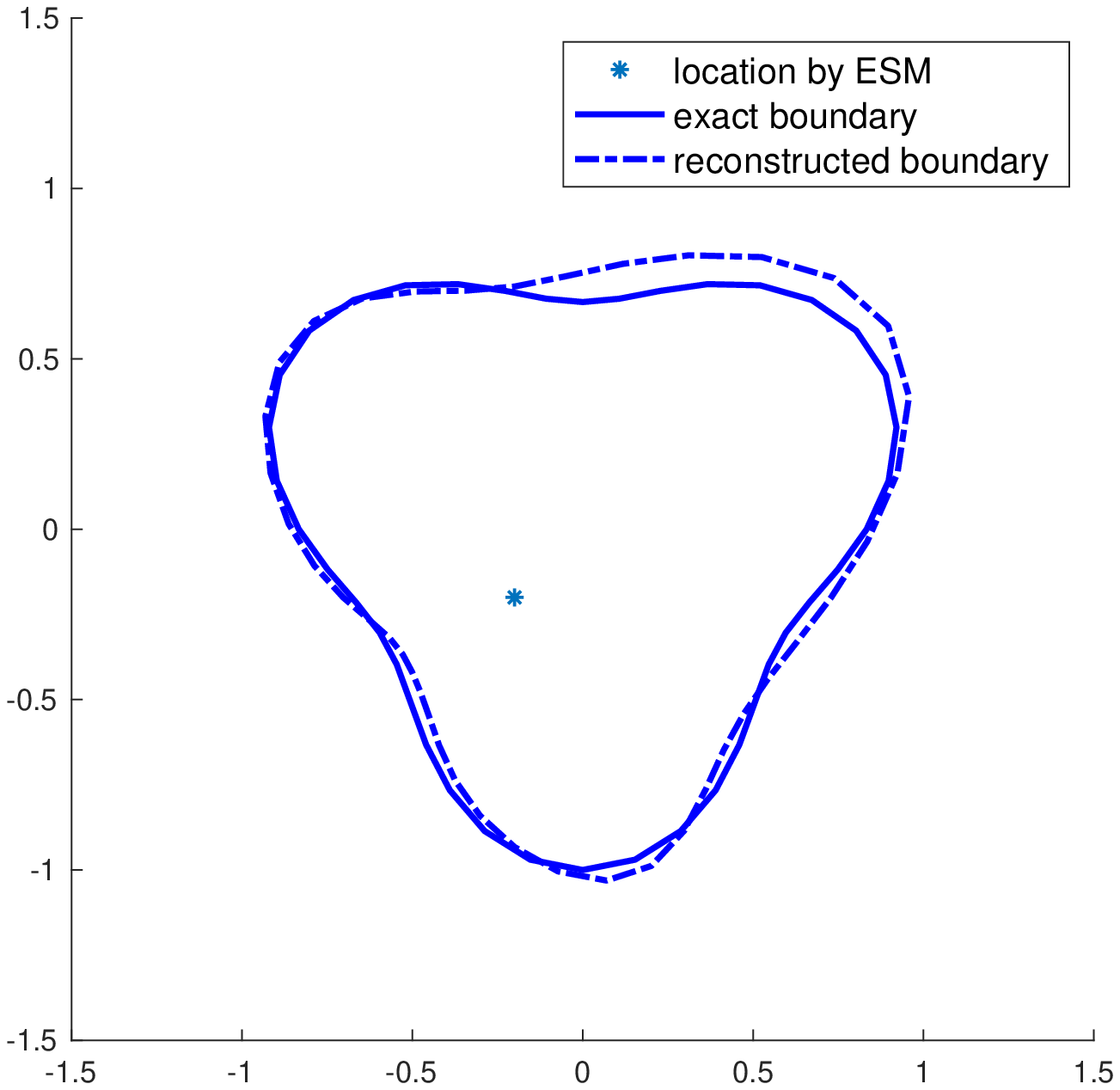}}&
\resizebox{0.52\textwidth}{!}{\includegraphics{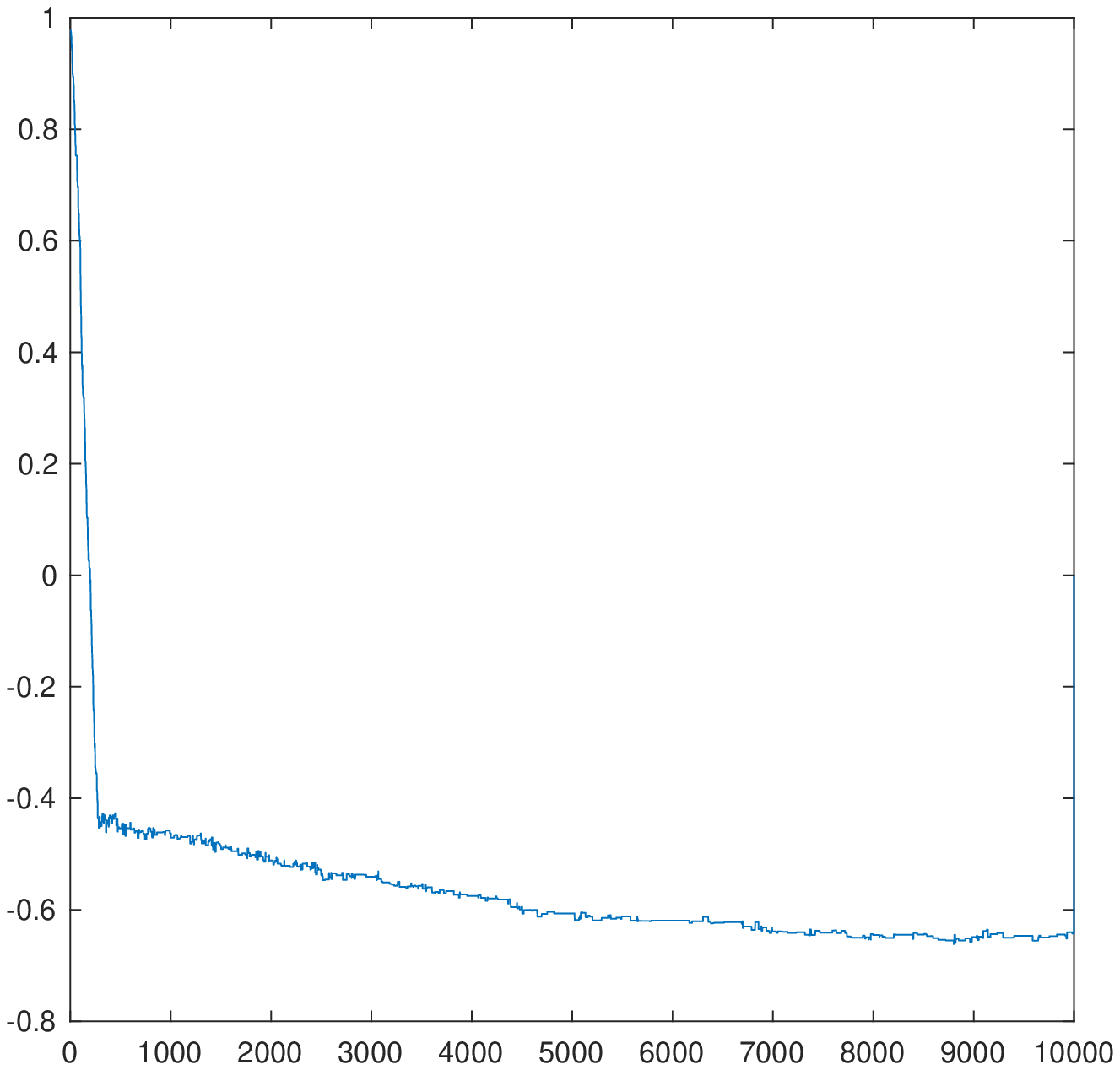}}\\
\resizebox{0.52\textwidth}{!}{\includegraphics{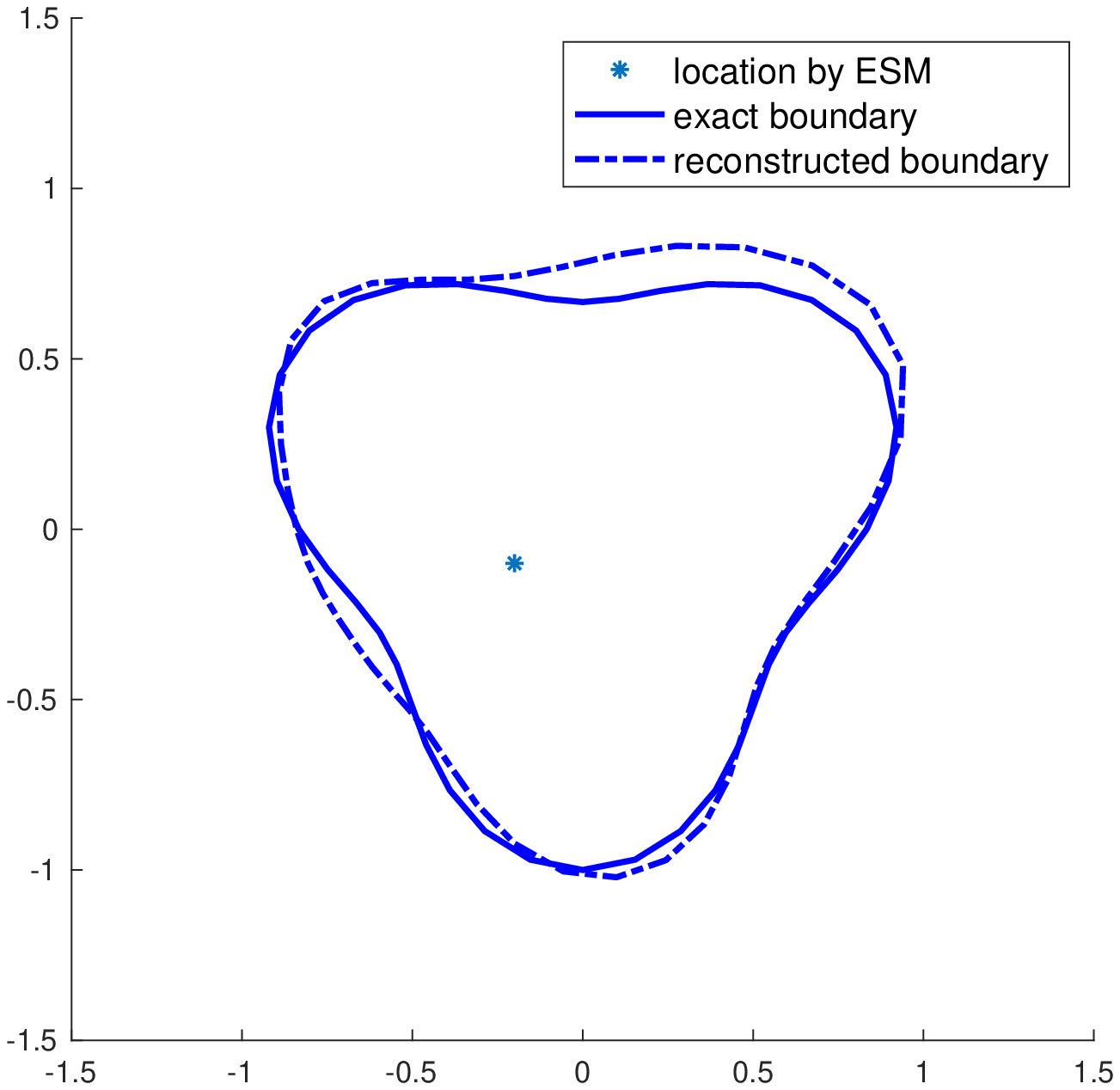}}&
\resizebox{0.52\textwidth}{!}{\includegraphics{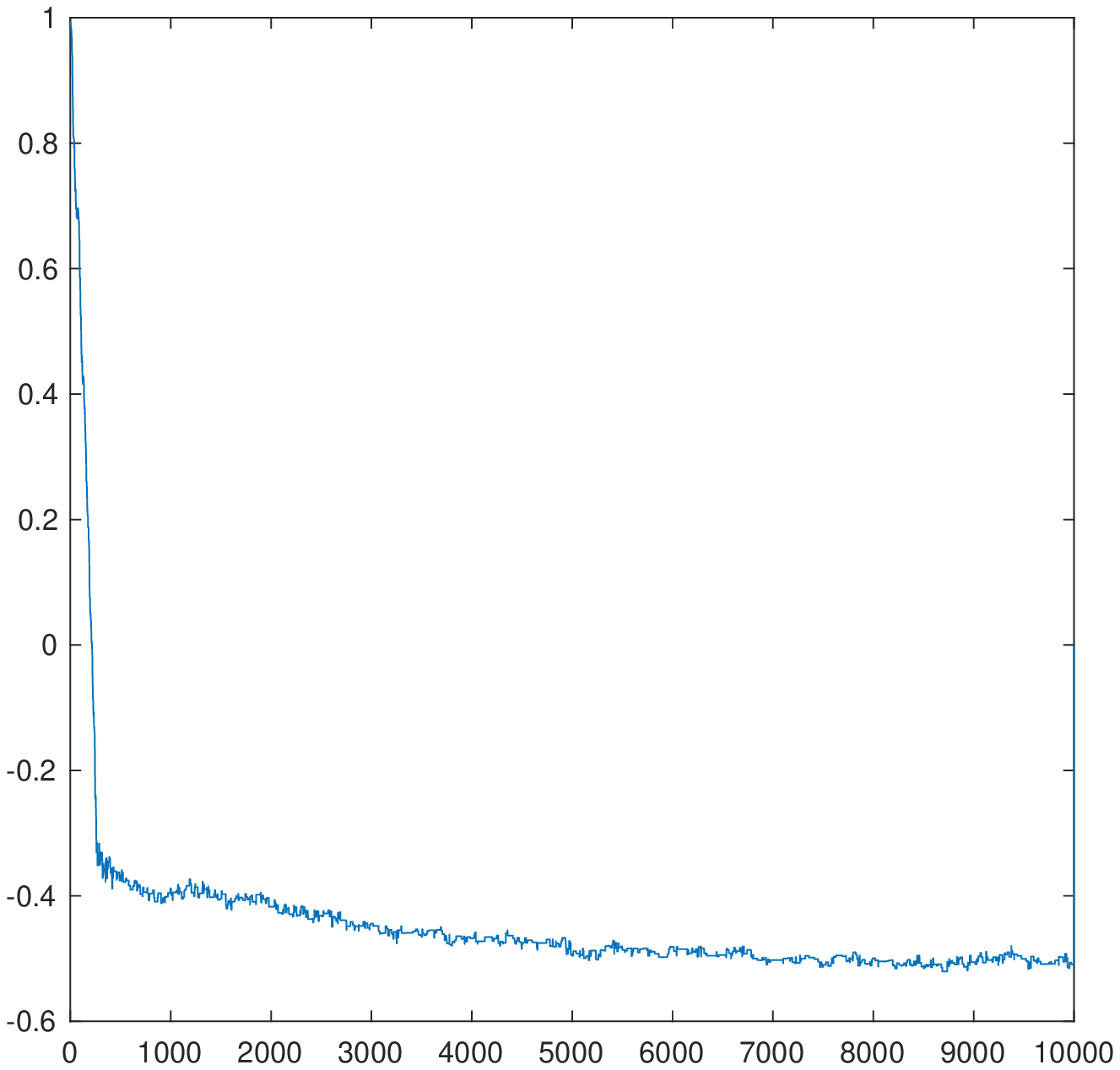}}
\end{tabular}
\end{center}
\caption{Reconstructions of the pear using multiple incident waves. Top left: $\gamma^o \times \gamma^i = \gamma_3^o \times \gamma^i_2$.
Top right: $\gamma^o \times \gamma^i = \gamma_4^o \times \gamma^i_2$. Bottom: $\gamma^o \times \gamma^i = \gamma_5^o \times \gamma^i_2$.}
\label{Fig4}
\end{figure}

\section{Conclusions}
We present some study of a Bayesian method for the limited aperture inverse scattering problems.
The extended sampling method is generalized to obtain the location of the obstacle, which is
critical to the fast convergence of the MCMC method.
Numerical examples show that the method can yield satisfactory reconstructions even when the measurement data is quite limited. 
The readers are encouraged to compare the results with other direct methods in inverse scattering using one incident wave, 
e.g., the extended sampling method \cite{LiuSun2018IP}. 

The Bayesian approach avoids to directly deal with the nonlinearity and the ill-posedness of the inverse problem,
but involves large computational cost. Several aspects will be investigated in the future to reduce the cost:
1. The faster algorithms for the direct scattering problem; 2. The better priors for the parametrization of the obstacle boundary; and
3. The more efficient transition kernel. Another interesting but challenging problem is the case of multiple obstacles.


\begin{thebibliography}{99}
\bibitem{AmmarEtal2013}H. Ammari, J. Garnier, W. Jing, H. Kang, M. Lim, K. S{\o}lna and H. Wang, 
	Mathematical and statistical methods for multistatic imaging. 
	Lecture Notes in Mathematics, Springer, Cham, 2013.
\bibitem{AmmariEtal2012SIAMIS} H. Ammari, J. Garnier, H. Kang, M. Lim and K. S{\o}lna,
	{\em Multistatic Imaging of Extended Targets.}
	SIAM J. Imaging Sci. 5(2), 564-600, 2012.
\bibitem{Atkinson1978} D. Atkinson, 
	Analytic extrapolations and inverse problems,
	Applied Inverse Problems (Lecture Notes in Physics 85) ed. P.C. Sabatier, Springer, Berlin, 111-121, 1978.
\bibitem{AhnJeonMaPark} C.Y. Ahn, K. Jeon, Y.K. Ma and W.K. Park,
	{\em A study on the topological derivative-based imaging of thin electromagnetic inhomogeneities in limited-aperture problems}.
	Inverse Problems 30, 105004, 2014.
\bibitem{AudibertHaddar2017SIAMIS} L. Audibert and H. Haddar,
	{\em 	The generalized linear sampling method for limited aperture measurements.}
	SIAM J. Imaging Sci. 10(2), 845-870, 2017.
\bibitem{BaoLiu2003SIAMSC} G. Bao and J. Liu,
	{\em Numerical solution of inverse problems with multi-experimental limited-aperture data}.
	SIAM J. Sci. Comput., {\bf25}(2003), 1102-1117.
\bibitem{BaussardEtal2001IP} A. Baussard, D. Pr\'{e}mel and O. Venard, 
	{\em A Bayesian approach for solving inverse scattering from microwave laboratory-controlled data.} 
	Special section: Testing inversion algorithms against experimental data. Inverse Problems 17, no. 6, 1659-1669, 2001.
\bibitem{Baum1999}  C.E. Baum, ed., Detection and Identification of Visually Obscured Targets, Taylor \& Francis, 1999.
\bibitem{BuiGhatts2014SIAMUQ} T. Bui-Thanh and O. Ghattas, 
	{\em An Analysis of Infinite Dimensional Bayesian Inverse Shape Acoustic Scattering and Its Numerical Approximation.} 
	SIAM/ASA Journal on Uncertainty Quantification 2, no. 1, 203-222, 2014.
\bibitem{CakoniColton2014} F. Cakoni and D. Colton, 
	A Qualitative Approach in Inverse Scattering Theory,
	AMS Vol.188, Springer-Verlag, 2014.
\bibitem{CalvettiKaipioSomersalo2014IP} D. Calvetti, J.P. Kaipio and E. Somersalo, 
	{\em Inverse problems in the Bayesian framework.} 
	Inverse Problems 30, no. 11, 110301, 2014.
\bibitem{ChengPengYamamoto2005IP} J. Cheng, L. Peng and M. Yamamoto, 
	{\em The conditional stability in line unique continuation for a wave equation and an inverse wave source problem.}
	Inverse Problems 21, 1993-2007, 2005.
\bibitem{ChengYamamoto1998IP} J. Cheng and M. Yamamoto, 
	{\em Unique continuation on a line for harmonic functions.}
	Inverse Problems 14, 869-882, 1998.
\bibitem{ColtonKirsch1996IP} D. Colton and A. Kirsch, 
	{\em A simple method for solving inverse scattering problems in the resonance region.}
	Inverse Problems 12, 383-393, 1996.
\bibitem{ColtonKress2013} D. Colton and R. Kress, 
	{\em Inverse Acoustic and Electromagnetic Scattering Theory} (Third Edition), Springer,  New York, 2013.
\bibitem{bib6a} B.T. Cox, S.R. Arridge and P.C. Beard, 
	{\em Photoacoustic tomography with a limited-aperture planar sensor and a reverberant cavity.} 
	Inverse Problems 23, no. 6, S95-S112, 2007.

\bibitem{Fitzpatrick1991IP} B.G. Fitzpatrick, 
	{\em Bayesian analysis in inverse problems.} 
	Inverse Problems 7, no. 5, 675-702, 1991.
\bibitem{GuoMonkColton2016AA} Y. Guo, P. Monk and Colton, 
	{\em The linear sampling method for sparse small aperture data.}
	Appl. Anal. 95, no. 8, 1599-1615, 2016.
\bibitem{HarrisRome2017AA} I. Harris and S. Rome, 
	{\em Near field imaging of small isotropic and extended anisotropic scatterers.} 
	Appl. Anal. 96, no. 10, 1713-1736, 2017.
\bibitem{HiptmairLi2018IP} R. Hiptmair and J. Li, 
	{\em Shape derivatives for scattering problems.} 
	Inverse Problems 34, no. 10, 105001, 2018.
\bibitem{IkehataNiemiSiltanen} M. Ikehata, E. Niemi and S. Siltanen,
	{\em Inverse obstacle scattering with limited-aperture data.}
	Inverse Probl. Imaging 1, 77-94, 2012.
\bibitem{ItoJinZou2012IP} K. Ito, B. Jin and J. Zou, 
	{\em A direct sampling method to an inverse medium scattering problem.}
	Inverse Problems 28, 025003, 2012.
\bibitem{KaipioSomerdalo2005} J. Kaipio and E. Somersalo, Statistical and Computational Inverse Problems, Applied Mathematical Sciences, Springer, 2005.
\bibitem{Juhl2008JASA} P. Juhl,
	{\em A comparison of SONAH and IBEM for near field acoustic holography.}
	J. Acoust. Soc. Am. 123, 2887-2892, 2008.
\bibitem{Kirsch98} A. Kirsch, 
	{\em Characterization of the shape of a scattering obstacle using the spectral data of the far field operator.} 
	Inverse Problems 14, 1489-1512, 1998.
\bibitem{Lewis1969IEEEAP} R. Lewis, 
	{\em Physical optics inverse diffraction.} 
	IEEE Transactions on Antennas \& Propagation 17, no. 3, 308-314, 1969.
\bibitem{LiuLiuSun2019IP} J. Liu, Y. Liu and J. Sun, 
	{\em An inverse medium problem using Stekloff eigenvalues and a Bayesian approach}, Inverse Problems, accepted.
\bibitem{LiuSun2018IP} J. Liu and J. Sun, 
	{\em Extended sampling method in inverse scattering.} 
	Inverse Problems 34, no. 8, 085007,  2018.
\bibitem{LiuLiuSun2019SIAMIS}  J. Liu, X. Liu and J. Sun,
	{\em Extended sampling method for inverse elastic scattering problems using one incident wave.}
	SIAM J. Imaging Sci. 12, no. 2, 874-892, 2019.
\bibitem{LiuSun2019JCP} X. Liu and J. Sun, 
	{\em Data recovery in inverse scattering: from limited-aperture to full-aperture.} 
	J. Comput. Phys. 386, 350-364, 2019.
\bibitem{OchsJr1987SIAMAM} R.L. Ochs, Jr.,
	{\em The limited aperture problem of inverse acoustic scattering: Dirichlet boundary conditions.}
	SIAM J. Appl. Math. 47, no. 6 , 1320-1341, 1987.
\bibitem{Potthast2006IP} R. Potthast,
	{\em A survey on sampling and probe methods for inverse problems.}
	Inverse Problems 22, no. 2, R1?R47, 2006.
\bibitem{Stuart2010AN} A. M. Stuart, 
	{\em Inverse problems: a Bayesian perspective.} 
	Acta Numerica 19, 451-559, 2010.
\bibitem{YangMaZheng2015} Y. Wang, F. Ma and E. Zheng, 
	{\em Bayesian method for shape reconstruction in the inverse interior scattering problem.} 
	Math. Probl. Eng., Art. ID 935294, 2015.
\bibitem{Zinn1989} A. Zinn,
	{\em On an optimisation method for the full- and limited-aperture problem in inverse acoustic scattering for a sound-soft obstacle.}
	Inverse Problems 5, 239-253, 1989.
\end{thebibliography}
\end{document}